\documentclass[twoside]{amsart}
\usepackage{amsmath,amssymb,amsthm}
\usepackage{verbatim}
\usepackage{paralist}
\usepackage{graphics} 
\usepackage{epsfig} 
\usepackage[colorlinks=true]{hyperref}
\hypersetup{urlcolor=blue, citecolor=red}

\newtheorem{theorem}{Theorem}[section]
\newtheorem{corollary}{Corollary}
\newtheorem{lemma}[theorem]{Lemma}
\newtheorem{proposition}{Proposition}

\theoremstyle{definition}
\newtheorem{definition}[theorem]{Definition}
\newtheorem{remark}{Remark}
\newcommand{\ep}{\varepsilon}

\newcommand{\B}{\mathcal{B}}
\newcommand{\D}{\mathcal{D}}
\newcommand{\E}{\mathcal{E}}
\newcommand{\F}{\mathcal{F}}
\newcommand{\G}{\mathcal{G}}
\newcommand{\R}{\mathbb{R}}
\newcommand{\N}{\mathbb{N}}
\newcommand{\dd }{\,\mathrm{d}}
\newcommand{\X}{X}

\newcommand{\EW}{\mathbb{E}}

\title{From Feynman-Kac Formulae to Numerical Stochastic Homogenization in Electrical Impedance Tomography}
\author{Petteri Piiroinen}
\address{Petteri Piiroinen\\ÊDepartment of Mathematics and Statistics\\  University of Helsinki \\ \phantom{Ho}FI-00014 Helsinki, Finland}
\email{petteri.piiroinen@helsinki.fi}
\author{Martin Simon}

\address{Martin Simon\\ÊInstitute of Mathematics\\ Johannes Gutenberg University\\ 55099 Mainz, \phantom{Ho}Germany} 
\email{simon@math.uni-mainz.de}
\date{\today}

\chardef\bslchar=`\\ % p. 424, TeXbook

\providecommand{\qedsymbol}{\leavevmode
  \hbox to.77778em{%
  \hfil\vrule
  \vbox to.675em{\hrule width.6em\vfil\hrule}%
  \vrule\hfil}}

\catcode`\|=0
\begingroup \catcode`\>=13 
\gdef\?#1>{{\normalfont$\langle$\textup{#1}$\rangle$}}
\gdef\0{\relax}
\endgroup
\def\<#1>{{\normalfont$\langle$\textup{#1}$\rangle$}}

\def\latex/{{\protect\LaTeX}}

\begin{document}
\maketitle
\markboth{Petteri Piiroinen and Martin Simon}{From Feynman-Kac formulae to stochastic homogenization in EIT}

%%%%%%%%%%%%%%%%%%%%%%%%%%%%%%%%%%%%%%%%%%%%%%%%%%%%%%%%%%%%%%%%%%%%%%%%
\begin{abstract}
In this paper, we use the theory of symmetric Dirichlet forms to derive Feynman-Kac formulae for the forward problem of electrical impedance tomography with possibly anisotropic, merely measurable conductivities corresponding to different electrode models on bounded Lipschitz domains. Subsequently, we employ these Feynman-Kac formulae to rigorously justify stochastic homogenization in the case of a stochastic boundary value problem arising from an inverse anomaly detection problem. 
Motivated by this theoretical result, we prove an estimate for the speed of convergence of the projected mean-square displacement of the underlying process which may serve as the theoretical foundation for the development of new scalable stochastic numerical homogenization schemes.
\end{abstract}

\section{Introduction}
Electrical impedance tomography (EIT) aims to reconstruct the unknown conductivity $\kappa$ in the conductivity equation 
\begin{equation}\label{eqn:con}
\nabla\cdot(\kappa\nabla u)=0\quad\text{in }D
\end{equation}
from current and voltage measurements on the boundary of the domain $D$. This inverse conductivity problem is known to be severely ill-posed, that is, its solution is extremely sensitive with respect to measurement and modeling errors. As a result, EIT suffers from inherent low resolution and due to this limitation, many practical applications focus on the detection of conductivity anomalies in a known background conductivity rather than conductivity imaging. In the mathematical modeling of such inverse anomaly detection problems, randomness typically reflects a lack of precise information about the meso- and microstructure of the heterogeneous background conductivity, which may fluctuate on many scales. Recently, the second author has proposed a novel method for the detection of conductivity anomalies in a random background conductivity which is based on homogenization of the underlying stochastic boundary value problem, cf. \cite{Simon}. 

Although the homogenization theory for elliptic divergence form operators is well-developed, cf., e.g., \cite{Bens,Pa,Pa1,Zikov1}, the numerical approximation of the \emph{effective conductivity} in the random setting still poses major challenges. The commonly used deterministic methods based on a discretization of the so-called \emph{auxiliary problem} have two main drawbacks. First, the auxiliary problem is formulated on the whole space $\R^d$ and second, it has to be solved for almost every realization of the random medium. That is, truncations of the auxiliary problem have to be considered and choosing an appropriate spatial truncation with appropriate boundary conditions is a delicate issue, cf. \cite{Bourgeat}. Moreover, in practically relevant cases, such as high contrast digitized random media, it is extremely difficult to solve the corresponding variational problems by usual deterministic methods, such as the finite element or the finite difference method, due to the behavior of the solutions near the corner points. Therefore, numerical approximation of the effective conductivity can be prohibitively expensive in terms of computation time. As a matter of fact, practitioners often choose to avoid these computations at all and rather content themselves with theoretical bounds, cf., e.g., \cite{Torquato}. However, it has been reported in the physical literature that the shortcomings of the standard deterministic methods can be circumvented by using continuum micro-scale Monte Carlo simulation of certain diffusion processes evolving in random media instead, cf., e.g., \cite{Keskin,Simonov,Torquato1,Torquato2,Torquato3}. In this work we give a rigorous mathematical justification for homogenizing the EIT forward problem using such methods by studying the interconnection between reflecting diffusion processes and certain boundary value problems for the conductivity equation. More precisely, we derive \emph{Feynman-Kac formulae} for solutions of the deterministic conductivity equation (\ref{eqn:con}) posed on a bounded domain $D\subset\mathbb{R}^d,d\geq 2$, with Lipschitz boundary $\partial D$ and possibly anisotropic uniformly elliptic and uniformly bounded conductivity subject to different boundary conditions modeling electrode measurements. Subsequently, we employ the Feynman-Kac formulae to prove a homogenization result for the corresponding stochastic boundary value problem which justifies the use of stochastic numerical homogenization schemes based on simulation of the underlying diffusion processes in order to approximate the effective conductivity. Finally, we prove an estimate for the speed of convergence of the projected mean-square displacement of the underlying diffusion processes. The main advantage of the presented approach to numerical homogenization, beside its inherent parallelism, is that its convergence rate is dimension-independent and its computational cost grows only linearly with the dimension.

It is well known, that reflecting diffusion processes generated by non-diver-gence form operators with smooth coefficients on 
bounded, smooth domains, are \emph{Feller processes} satisfying Skorohod type stochastic differential equations. 
The construction in the case of divergence form operators with merely measurable
coefficients requires the theory of \emph{symmetric Dirichlet forms} which has its origin in the \emph{energy method} 
used by Dirichlet to address the boundary value problem in classical electrostatics that was subsequently named after him. When $D$ is a bounded Lipschitz domain, Bass and Hsu~\cite{Bass}
constructed the reflecting Brownian motion living on $\overline{D}$ by
showing that the so-called \emph{Martin-Kuramochi boundary} coincides with the
Euclidean boundary in this case. A general diffusion process on a
bounded Lipschitz domain, even allowing locally a finite number of
H\"older cusps, was first constructed by Fukushima and Tomisaki
\cite{FukushimaTomisaki}. In this work, we use such a Dirichlet form construction in order to derive
Feynman-Kac representation formulae for the solutions of 
Neumann, respectively Robin, boundary value problems modeling EIT measurements. Probabilistic approaches to both, parabolic and elliptic boundary value
problems for second order differential operators have been studied by
many authors, starting with Feynman's Princeton thesis \cite{Feynman}
and the article \cite{Kac} by Kac. The probabilistic approach to the
Dirichlet problem for a general class of second-order elliptic operators
with merely measurable coefficients, even allowing singularities of a
certain type, was elaborated by Chen and Zhang \cite{ChenZhang}; see
also Zhang's paper \cite{Zhang}. However, there are only few works that
treat Feynman-Kac representation formulae for Neumann or Robin type
boundary conditions.  Moreover, the approaches existing in the literature
consider either the Laplacian, see, e.g., \cite{Bass,Brosamler,Hsu}, or
non-divergence form operators with smooth coefficients, see, e.g.,
\cite{Freidlin, Papanicolaou, Pardoux}. 
For the particular case of the conductivity equation on a
bounded Lipschitz domain, we generalize both, the Feynman-Kac formula
for the Robin problem on domains with boundary of class $C^3$ for an
isotropic $C^{2,\gamma}$-smooth conductivity, $\gamma>0$, obtained by
Papanicolaou \cite{Papanicolaou} as well as the representation
obtained by Bench{\'e}rif-Madani and Pardoux \cite{Pardoux} for the
Neumann problem under similar regularity assumptions. While both of
the aforementioned approaches use stochastic
differential equations and It\^{o} calculus, our approach is based on the 
theory of symmetric Dirichlet forms, following the pioneering work \cite{Bass} for the reflecting Brownian motion
by Bass and Hsu.  We derive in this work Feynman-Kac formulae for both, the Robin boundary value problem, 
corresponding to the so-called \emph{complete electrode model}, as well as the Neumann boundary value problem, corresponding to the so-called 
\emph{continuum model}. Both formulae are
valid for possibly anisotropic, uniformly elliptic and uniformly bounded conductivities with 
merely measurable coefficients on bounded Lipschitz domains. During the preparation of this work we became aware of the paper
\cite{ChenZhang2} by Chen and Zhang, where a probabilistic approach to
some mixed boundary value problems with singular coefficients is
derived.  In contrast to our setting, however, the mixed boundary
condition studied there results from a singular lower-order
term of the differential operator.

Homogenization of reflected stochastic differential equations and partial differential equations with Neumann boundary conditions, respectively, in half- space type domains have been studied for periodic coefficients in \cite{Barles,Bens,Tanaka} and for random divergence form operators with smooth coefficients in \cite{Rh}. In contrast to boundary value problems with homogeneous Dirichlet boundary conditions, these problems are non-translation invariant, which excludes the standard stochastic homogenization approach via the so-called \emph{environment as viewed from the particle}. Employing the Feynman-Kac formula in conjunction with a recently obtained invariance principle for reflecting diffusion processes associated with random divergence form operators with merely measurable coefficients due to Chen, Croydon and Kumagai \cite{Chen1} we provide a homogenization result for a stochastic forward problem built on the complete electrode model. 
Clearly, such a result motivates the derivation of stochastic numerical homogenization schemes for the approximation of the effective conductivity which are based on simulation of the underlying diffusion processes. However, the convergence analysis of such a method requires a quantitative convergence result that is stronger than the usual qualitative results obtained from the central limit theorem for martingales. As in the case of a discrete random walk in random environment, cf. Gloria and Mourrat \cite{Gloria}, it turns out that the behavior at the bottom of the spectrum of the infinitesimal generator of the environment as viewed from the particle process, projected on a suitably chosen function, yields bounds on the approximation error. This spectral behavior has been the subject of recent interest. Most notably, Gloria, Neukamm and Otto \cite{GloriaNeukamm} have obtained optimal estimates in the discrete case which have been carried over to the continuum case by Gloria and Otto \cite{GloriaOtto}. The main difficulty in obtaining such estimates for diffusion processes evolving in random media arises from the lack of a Poincar\'{e} inequality for the horizontal derivative in the space of square integrable functions on the probability space which corresponds to the random medium. Therefore, in contrast to the periodic case, where the Poincar\'e inequality on the torus is available, one can not expect a spectral gap in the random case. Still, it has been shown that the bottom of the spectrum is sufficiently \lq\lq thin\rq\rq. Using these estimates together with a classical argument due to Kipnis and Varadhan \cite{Kipnis}, we obtain an estimate for the speed of convergence of the projected mean-square displacement of the underlying diffusion process in a random medium to its limit. Qualitative results of this kind have been obtained by Kipnis and Varadhan in the case of discrete random walks in random environments and by De Masi, Ferrari, Goldstein and Wick \cite{DeMasi} in the continuum case, whereas qualitative results in the case of discrete random walks have been proved more recently by Gloria and Mourrat \cite{Gloria} and Egloffe, Gloria, Mourrat and Nguyen \cite{Egloffe}. Finally, we refer to the paper \cite{Mourrat} by Mourrat which initiated the idea of using the Kipnis and Varadhan argument in order to obtain quantitative results.

The rest of the paper is structured as follows: We start in Section \ref{section2} by 
briefly introducing our notation. In Section \ref{section1a}, we recall the modeling of electrode measurements in EIT as well as the modeling of random heterogeneous media. Moreover, we introduce the stochastic forward problem we are interested in. In Section \ref{section3}, we describe the construction of reflecting diffusion processes via Dirichlet form theory and in Section \ref{sectionSk} we derive Skorohod decompositions for two practically relevant classes of conductivities. Subsequently, in Section \ref{section4}, the Feynman-Kac formulae for the deterministic boundary value problems will be derived. Then in Section \ref{section5} we study the interconnection between Feynman-Kac formulae, stochastic homogenization and stochastic numerics. Finally, we conclude with a brief summary of our results.

%%%%%%%%%%%%%%%%%%%%%%%%%%%%%%%%%%%%%%%%%%%%%%%%%%%
\section{Notation}
\label{section2}
Let $D$ denote a bounded Lipschitz domain in $\R^d$, $d\geq 2$, with connected complement and Lipschitz parameters
$(r_D,c_D)$, i.e., for every $x\in\partial D$ we have after rotation and translation that $\partial D\cap B(x,r_D)$ is the graph of a Lipschitz function in the first $d-1$ coordinates with Lipschitz constant no larger than $c_D$ and $D\cap B(x,r_D)$ lies above the graph of this function. Moreover, we set $\R^d_-:=\{x\in\R^d: x\cdot \nu<0\}$, with $\nu=e_d$ the outward unit normal on $\R^{d-1}$, where we identify the boundary of $\R^d_-$ with $\R^{d-1}$, with straightforward abuse of notation. 

For Lipschitz domains, there exists a unique outward unit normal vector $\nu$ a.e. on $\partial D$ so that the real Lebesgue spaces $L^p(D)$ and $L^p(\partial D)$ can be defined in the standard manner with the usual $L^p$ norms $\lvert\lvert\cdot\rvert\rvert_p$, $p=1,2,\infty$. The standard $L^2$ inner-products are denoted by $\left<\cdot,\cdot\right>$ and $\left<\cdot,\cdot\right>_{\partial D}$, respectively. The $d$-dimensional Lebesgue measure is denoted by $m$, the $(d-1)$-dimensional Lebesgue surface measure is denoted by $\sigma$ and $\lvert\cdot\rvert$ denotes the Euclidean norm on $\R^d$. 

By $(\Gamma,\G,\mathcal{P})$ we always mean a complete probability space corresponding to a random medium. We use the notation $\omega$ for an arbitrary element of $\Gamma$ and $\mathbb{M}$ for the expectation with respect to the probability measure $\mathcal{P}$. We use bold letters to denote functions on $(\Gamma,\G,\mathcal{P})$, while we use italic letters for the corresponding realizations on $\R^d\times\Gamma$.
The canonical probability space corresponding to diffusion processes evolving in a deterministic medium starting in $x$ is denoted $(\Omega,\F,\mathbb{P}_x)$ and the expectation with respect to $\mathbb{P}_x$ is denoted $\EW_x$. If the process is evolving in a random medium, we indicate this with a superscript $\omega$ for the probability measure, i.e., the measure $\mathbb{P}_x^{\omega}$ corresponds to the particular realization $\omega$ of the medium.
Finally, the product probability space corresponding to the \emph{annealed measure }Ê$\overline{\mathbb{P}}:=\mathcal{P}\,\mathbb{P}_0^{\omega}$ on $\overline{\Omega}:=\Gamma\times\Omega$, which is obtained by integrating with respect to the measure $\mathbb{P}_0^{\omega}$ and subsequent averaging over the realizations of the random medium, is denoted $(\overline{\Omega},\overline{\F},\overline{\mathbb P})$. The expectation with respect to $\overline{\mathbb{P}}$ is denoted $\overline{\EW}$.

All functions in this work will be real-valued and derivatives are understood in distributional sense. We use a diamond subscript to denote subspaces of the standard Sobolev spaces containing functions with vanishing mean and interpret integrals over $\partial D$ as dual evaluations with a constant function, if necessary. For example, we will frequently use the spaces
\begin{equation*}
H^{\pm 1/2}_{\diamond}(\partial D):=\Big\{\phi\in H^{\pm 1/2}(\partial D):\left<\phi,1\right>_{\partial D}=0\Big\}
\end{equation*}
and
\begin{equation*}
H^{1}_{\diamond}(D):=\Big\{\phi\in H^{1}(D):\left<\phi,1\right>=0\Big\}.
\end{equation*}

Moreover, we will frequently assume that $\partial D$ is partitioned into two disjoint parts, $\partial_1 D$ and $\partial_2 D$. We denote by $H_0^1(D\cup\partial_1 D)$ the closure of $C_c^{\infty}(D\cup\partial_1 D)$, the linear subspace of $C^{\infty}(\overline{D})$ consisting of functions $\phi$ such that $\mathrm{supp}(\phi)$ is a compact subset of $D\cup\partial_1 D$, in $H^1(D)$. Moreover, we define the Bochner space $L^2(\Gamma;H^1_{0}(D\cup\partial_1 D))$ 
\begin{equation*}
=\Big\{\boldsymbol{\phi}:\Gamma\rightarrow H^1_{0}(D\cup\partial_1 D):\int_{\Gamma}\lvert\lvert \boldsymbol{\phi}(\cdot,\omega)\rvert\rvert^2_{H^1_{0}(D\cup\partial_1 D)}\dd\mathcal{P}(\omega)<\infty\Big\},
\end{equation*}
see, e.g.,~\cite{Babuska} for properties of this space. 

For the reason of notational compactness, we use the Iverson brackets: Let $S$ be a mathematical statement, then 
\begin{equation*}\left[S\right]=\begin{cases}
1,\quad&\text{if }S\text{ is true}\\
0,\quad&\text{otherwise}.
\end{cases}
\end{equation*}
We also use the Iverson brackets $[x\in B]$ to denote the indicator function of a set $B$, which we abbreviate by $[B]$ if there is no danger of confusion. 

In what follows, all unimportant constants are denoted $c$, sometimes with additional subscripts, and they may vary from line to line. 

%%%%%%%%%%%%%%%%%%%%%%%%%%%%%%%%%%%%%%%%%%%%%%%%%%%

\section{Electrical impedance tomography forward problems}\label{section1a}
\subsection{Modeling of electrode measurements}
We assume that the, possibly anisotropic, conductivity is defined by a symmetric,
matrix-valued function $\kappa:D\rightarrow \mathbb{R}^{d\times d}$ with
components in $L^{\infty}(D)$ such that $\kappa$ is uniformly bounded and uniformly elliptic, i.e., there exists some constant $c>0$ such that
\begin{equation}\label{eqn:ellipticity}
c^{-1}\lvert\xi\rvert^2\leq \xi\cdot\kappa (x)\xi\leq
c\lvert\xi\rvert^2,\quad \text{for every
}\xi\in\mathbb{R}^d\text{ and a.e. }x\in D.  
\end{equation}

The forward problem of electrical impedance tomography can be described by different measurement
models. In the so-called \textit{continuum model}, the conductivity equation (\ref{eqn:con}) is equipped with a co-normal boundary condition 
\begin{equation}\label{eqn:continuum}
\partial_{\kappa\nu}u:= \kappa\nu\cdot\nabla u\vert_{\partial D}=f\quad\text{on }\partial D,
\end{equation}
where $f$ is a measurable function modeling the signed density of the outgoing current. The boundary value problem (\ref{eqn:con}), (\ref{eqn:continuum}) has a solution if and only if 
\begin{equation}\label{eqn:compatibility}
\left<f,1\right>_{\partial D}=0.
\end{equation}
Physically speaking, this means that the current must be conserved. Given an appropriate function $f$, the solution to (\ref{eqn:con}), (\ref{eqn:continuum}) is unique up to an additive constant, which physically corresponds to the choice of the ground level of the potential. If $f\in H^{-1/2}_{\diamond}$, then there exists a unique equivalence class of functions $u\in H^1(D)/\R$ that satisfies the weak formulation of the boundary value problem
\begin{equation*}
\int_D\kappa\nabla u\cdot\nabla v\dd x=\left <f, v\vert_{\partial D}\right >_{\partial D}\quad \text{for all }v\in H^1(D)/\R,
\end{equation*} 
where $v\vert_{\partial D}:=\gamma v$ and $\gamma:H^1(D)/\R\rightarrow H^{1/2}(\partial D)/\R=(H^{-1/2}_{\diamond}(\partial D))'$ is the standard trace operator. Note that we occasionally write $v$ instead of $v\vert_{\partial D}$ for the sake of readability.

In practical EIT measurement, a number of electrodes, denoted $E_1,...,E_N\subset\partial D$, are attached on the boundary of the object $D$. These electrodes are modeled by disjoint surface patches given by simply connected subsets of $\partial D$, each having a Lipschitz boundary curve. 
The most accurate forward model for real-life EIT is the so-called \textit{complete electrode model} which takes into account the fact that during electrode measurements there is a \textit{contact impedance }caused by a thin, highly resistive layer at the electrode object interface. It was demonstrated experimentally that the complete electrode model can correctly predict measurements up to instrument precision, cf. \cite{Somersalo}. 
For a given voltage pattern $U\in \R^N$ the boundary conditions for the complete electrode model are given by
\begin{equation}\label{eqn:cem}
\kappa\nu\cdot\nabla u\vert_{\partial D}+gu\vert_{\partial D}=f\quad\text{on}\ \partial D, 
\end{equation}
where the functions $f,g:\partial D\rightarrow\mathbb{R}$ are defined by
\begin{equation}\label{eqn:functions}
f(x):=\frac{1}{z(x)}\sum_{l=1}^NU_l[E_l],\quad g(x):=\frac{1}{z(x)}\sum_{l=1}^N [E_l]
\end{equation}
and the contact impedance $z:\partial D\rightarrow\R$ is assumed to be a
piecewise continuous function, with interfaces that are of zero surface
measure, satisfying
\begin{equation*}
0<c_0\leq z\leq c_1\quad\text{a.e. on }\partial D.
\end{equation*}
For a given voltage pattern $U\in\mathbb{R}^N$ satisfying the grounding condition
\begin{equation}\label{eqn:ground}
\sum_{l=1}^N U_l = 0,
\end{equation}
the equations (\ref{eqn:con}) and (\ref{eqn:cem}) define the electric
potential $u\in H^1(D)$ uniquely, cf.~\cite{Somersalo} and the variational form of the boundary value problem (\ref{eqn:con}), (\ref{eqn:cem}) reads as follows: Given $U\in\R^N$ satisfying (\ref{eqn:ground}), find $u\in H^1(D)$ such that 
\begin{equation}
\int_D\kappa\nabla u\cdot\nabla v\dd x+\left<gu\vert_{\partial D},v\vert_{\partial D}\right>_{\partial D}=\left<f, v\vert_{\partial D}\right>_{\partial D}\quad\text{for all }v\in H^1(D).
\end{equation}
Knowledge of $u$ yields the corresponding electrode current vector $J\in\R^N$ via
\begin{equation}\label{eqn:cem0}
J_l=\int_{E_l}\partial_{\kappa\nu} u\dd\sigma(x),\quad 1\leq l\leq N.
\end{equation}

\subsection{A stochastic forward problem}\label{sec:stochan}\label{subsec:proba}
The basic geometric setting of the stochastic problem we are interested in is as follows: Assume that the model domain is given by the lower hemisphere $$D:=B(0,R)\cap\R^d_-,\quad R>0$$ and that $\partial D$ is partitioned into two disjoint parts, namely the \emph{accessible boundary }$\partial_1 D:=\partial D\cap \R^{d-1}$ and the \emph{inaccessible boundary }$\partial_2 D:=\partial D\backslash\partial_1 D$, respectively. Such a setting is found for instance in geophysical applications, where measurements can only be taken on the surface, cf., e.g., \cite{Simon}.

Let $(\Gamma,\G,\mathcal{P})$ be a probability space and let $\boldsymbol{\Theta}:\Gamma\rightarrow\Gamma$ denote an ergodic $d$-dimensional \emph{dynamical system}, i.e., a family of automorphisms $\{\boldsymbol{\Theta}_x,x\in\R^d\}$ which satisfies the following conditions:
\begin{enumerate}
\item[(i)] The family $\{\boldsymbol{\Theta}_x,x\in\R^d\}$ is a group, i.e., $\boldsymbol{\Theta}_0=\operatorname{id}$ and
\begin{equation*}
\boldsymbol{\Theta}_{x+y}=\boldsymbol{\Theta}_x\boldsymbol{\Theta}_y\quad\text{for all }x,y\in\R^d;
\end{equation*} 
\item[(ii)] the mappings $\boldsymbol{\Theta}_x:\Gamma\rightarrow\Gamma$, $x\in\R^d$, preserve the measure $\mathcal{P}$ on $\Gamma$, i.e., for every $B\in\G$, $\boldsymbol{\Theta}_x B$ is $\mathcal{P}$-measurable and
\begin{equation*}
\mathcal{P}(\boldsymbol{\Theta}_xB)=\mathcal{P} (B);
\end{equation*}
\item[(iii)] for every measurable function $\boldsymbol{\phi}$ on $(\Gamma,\G,\mathcal{P})$, the function $(x,\omega)\mapsto\boldsymbol{\phi}(\boldsymbol{\Theta}_x\omega)$ is a measurable function on $(\R^d\times\Gamma,\B(\R^d)\otimes\G,m\times\mathcal{P})$, where $\B(\R^d)\otimes\G$ denotes the sigma-algebra generated by the measurable rectangles;
\item[(iv)] the family $\{\boldsymbol{\Theta}_x,x\in\R^d\}$ is ergodic, i.e., $\boldsymbol{\phi}(\boldsymbol{\Theta}_x\omega)=\boldsymbol{\phi}(\omega)$ for all $x\in\R^d$ and $\mathcal{P}$-a.e. $\omega\in\Gamma$ implies $\boldsymbol{\phi}=\operatorname{const}$ $\mathcal{P}$-a.e.
\end{enumerate}

Throughout this work we assume that the conductivity random field $$\{\kappa(x,\omega),(x,\omega)\in\R^d\times\Gamma\}$$ is the \emph{stationary extension} on $\R^d\times\Gamma$Êof some function $\boldsymbol{\kappa}\in L^{2}(\Gamma;\R^{d\times d})$, that is,
\begin{equation}\label{eqn:stat_ex}
(x,\omega)\mapsto\kappa(x,\omega)=\boldsymbol{\kappa}(\boldsymbol{\Theta}_x\omega).
\end{equation} 
Note that if  $\kappa$ can be written in the form (\ref{eqn:stat_ex}) with a dynamical system 
$\{\boldsymbol{\Theta}_x,x\in\R^d\}$ which satisfies (i)-(iii), then it is automatically \emph{stationary} with respect to $\mathcal{P}$, i.e., for every finite collection of points $x^{(i)}$, $i=1,...,k$, and any $h\in\R^d$ the joint distribution of 
\begin{equation*}
\kappa(x^{(1)}+h,\omega),...,\kappa(x^{(k)}+h,\omega)
\end{equation*}
under $\mathcal{P}$ is the same as that of
\begin{equation*}
\kappa(x^{(1)},\omega),...,\kappa(x^{(k)},\omega).
\end{equation*}
Even if it is not explicitly stated, we always assume that the conductivity random field may be written in the form (\ref{eqn:stat_ex}), where the underlying dynamical system 
$\{\boldsymbol{\Theta}_x,x\in\R^d\}$ satisfies conditions (i)-(iv).
Moreover, we will explicitly state if the conductivity random field $\{\kappa(x,\omega),(x,\omega)\in\R^d\times\Gamma\}$ satisfies one of the following assumptions:
\begin{enumerate}
\item[(A1)] $\boldsymbol{\kappa}\in L^{2}(\Gamma;L^{\infty}(\R^d;\R^{d\times d}))$ and the random field is strictly positive and uniformly bounded, that is, there exists a constant $c>0$ such that for every $\xi\in\R^d$ and a.e. $x\in \R^d$
\begin{equation*}
\mathcal{P}\Big\{\omega\in\Gamma: c^{-1} \lvert \xi\rvert^2\leq \xi\cdot\kappa(x,\omega)\xi\leq c\lvert \xi\rvert^2\Big\}=1.
\end{equation*}
\item[(A2)] $\{\kappa(x,\omega), (x,\omega)\in\R^d\times\Gamma\}$ satisfies the \emph{spectral gap property}, cf. \cite{GloriaOtto}: There exist constants $\rho>0$ and $r<\infty$ such that for all measurable functions on $\Big\{\kappa:\mathbb{R}^d\rightarrow\{\kappa_0\in\mathbb{R}^{d\times d}:\lvert\kappa_0\xi\rvert\leq\lvert\xi\rvert,c\lvert\xi\rvert^2\leq\xi\cdot\kappa_0\xi\text{ for all }\xi\in\mathbb{R}^d\}\Big\}$
\begin{equation*}
\mathbb{V}\boldsymbol{\phi}\leq\frac{1}{\rho}\mathbb{M}\int_{\mathbb{R}^d}\Big(\operatorname{osc}_{\kappa\vert_{B(x,r)}}\boldsymbol{\phi}\Big)^2\,\mathrm{d}x,
\end{equation*}
where we have set 
\begin{equation*}\begin{split}
\Big(\operatorname{osc}_{\kappa\vert_{B(x,r)}}\boldsymbol{\phi}\Big):=&\sup\Big\{\boldsymbol{\phi}(\tilde{\kappa}):\tilde{\kappa}\in\Omega,\tilde{\kappa}\vert_{\mathbb{R}^d\backslash B(x,r)}=\kappa\vert_{\mathbb{R}^d\backslash B(x,r)}\Big\} \\
&-\inf\Big\{\boldsymbol{\phi}(\tilde{\kappa}):\tilde{\kappa}\in \Omega,\tilde{\kappa}\vert_{\mathbb{R}^d\backslash B(x,r)}=\kappa\vert_{\mathbb{R}^d\backslash B(x,r)}\Big\}.
\end{split}
\end{equation*}
\end{enumerate}

To account for the highly heterogeneous properties of the background medium, the latter is modeled using the conductivity random field with appropriate scaling by a small parameter $\ep>0$, i.e.,
\begin{equation*}
\kappa_{\ep}(\cdot,\cdot):\R^d\times \Gamma\rightarrow\R^{d\times d},\quad\kappa_{\ep}(x,\omega):=\kappa(x/\ep,\omega).
\end{equation*}
If the correlation length of the conductivity random field $\kappa$ is, say $1$, then the correlation length of the scaled version $\kappa_{\ep}$ is of order $\ep$ and for $\ep\ll 1$ we obtain thus a rapidly oscillating random field.
\subsection{Stochastic forward and inverse problem}

Now let us introduce a stochastic forward model based on the complete electrode model: We search for a random field $\{\boldsymbol{u}_{\ep}(x,\omega),(x,\omega)\in\overline{D}\times\Gamma\}$ for the electrical potential with $u_{\ep}\in L^2(\Gamma;H^1_{0}(D\cup\partial_1 D))$
such that the stochastic conductivity equation
\begin{equation}\label{eqn:random_cond}
\nabla\cdot(\kappa_{\ep}\nabla u_{\ep})=0\quad\text{in }D\times\Gamma
\end{equation}
subject to the boundary conditions
\begin{equation}\label{eqn:fp2}
\begin{aligned}
&\kappa_\ep\nu\cdot\nabla u_{\ep}\vert_{\partial_1 D}+gu_{\ep}\vert_{\partial_1 D}=f\quad&\text{on }\partial_1 D\times\Gamma\\
&u_{\ep}\vert_{\partial_2 D}=0\quad&\text{on }\partial_2 D\times\Gamma
\end{aligned}
\end{equation}
is satisfied $\mathcal{P}$-a.s. 
The variational formulation of the forward problem is to find $\boldsymbol{u}_{\ep}\in L^2(\Gamma;H^1_{0}(D\cup\partial_1 D))$ such that 
\begin{equation}\label{eqn:variational1}
\mathbb{M}\Big\{\int_{D}\kappa_{\ep}\nabla u_{\ep}\cdot\nabla v\dd x+\left<g u_{\ep},v\right>_{\partial_1 D}\Big\}=\mathbb{M}\left<f,v\right>_{\partial_1 D}
\end{equation} 
for all $\boldsymbol{v}\in L^2(\Gamma;H^1_{0}(D\cup\partial_1 D))$.
For a given voltage pattern $U\in\R^N$ the corresponding measurement data is given by the random current pattern $J(\ep,\omega)=(J_1(\ep,\omega),...,J_N(\ep,\omega))^T$, defined for $\mathcal{P}$-a.e. $\omega\in\Gamma$ by
\begin{equation}\label{eqn:curr}
J_l(\ep,\omega)=\frac{1}{\lvert E_l\rvert}\int_{E_l}\kappa_{\ep}(x,\omega)\nu\cdot\nabla u_{\ep}(x,\omega)\vert_{\partial_1 D}\dd \sigma(x),\quad l=1,...,N.
\end{equation}
Due to the assumption (A1), the well-posedness of the variational formulation follows from a straightforward application of the Lax-Milgram theorem. Moreover, standard arguments from measure theory show that the solution to the stochastic forward problem (\ref{eqn:random_cond}), (\ref{eqn:fp2}) also solves (\ref{eqn:variational1}), cf.~\cite{Babuska}.

%%%%%%%%%%%%%%%%%%%%%%%%%%%%%%%%%%%%%%%%%%%%%%%%%%%

\section{Reflecting diffusion processes}\label{section3}
In his seminal paper \cite{Fukushima}, Fukushima established a
one-to-one correspondence between regular symmetric Dirichlet forms and
symmetric Hunt processes, which is the foundation for the construction of stochastic processes via 
Dirichlet form techniques. Therefore we assume that the reader is familiar with the theory of symmetric Dirichlet forms, as elaborated for instance in the monograph \cite{Fukushimaetal}.

Let us consider the following symmetric bilinear forms on $L^2(D)$:
\begin{equation}\label{eqn:Dirichlet}
\E(v,w):=\int_{D}\kappa\nabla v(x)\cdot\nabla w(x)\dd  x,\quad
v,w\in\D(\E):=H^1(D)
\end{equation}
and for the particular case $\kappa\equiv 1/2$, which is of special importance, we set
\begin{equation}\label{eqn:Dirichlet1}
\E^{\text{BM}}(v,w):=\frac{1}{2}\int_{D}\nabla v(x)\cdot\nabla w(x)\dd  x,\quad
v,w\in\D(\E^{\text{BM}}):=H^1(D).
\end{equation} 

The pair $(\E, \D(\E))$ defined by (\ref{eqn:Dirichlet}) is a strongly local, regular symmetric Dirichlet form on $L^2(D)$. In particular, there exist an $\E$-exceptional set
$\mathcal{N}\subset\overline{D}$ and a conservative diffusion process
$\X=(\Omega,\F,\{\X_t,t\geq 0\},\mathbb{P}_{x})$, starting from 
$x\in\overline{D}\backslash \mathcal{N}$ such that 
$\X$ is associated with $(\E,\D(\E))$. Without loss of generality let us assume that $\X$ is defined on the \emph{canonical
sample space} $\Omega=C([0,\infty);\overline{D})$. 
It is well-known that the symmetric Hunt process associated with (\ref{eqn:Dirichlet1}) is the
reflecting Brownian motion. Therefore, we call the symmetric Hunt process associated with 
(\ref{eqn:Dirichlet}) a \emph{reflecting diffusion process}. 

Let us briefly recall the concept of the \emph{boundary local time }of reflecting diffusion processes, see, e.g., \cite{Hsu,Papanicolaou,Pardoux}, which will be crucial for the subsequent derivation of the Feynman-Kac formulae. If the diffusion process is the solution to a stochastic differential equation, say the reflecting Brownian motion, then the boundary local time is given by the one-dimensional process $L$ in the Skorohod decomposition, which prevents the sample paths from leaving $\overline{D}$, i.e.,
\begin{equation}\label{eqn:SkorohodRBM}
\X_t=x+W_t-\frac{1}{2}\int_0^t\nu(\X_s)\dd L_s, 
\end{equation} 
$\mathbb{P}_x$-a.s. for q.e. $x\in\overline{D}$. This boundary local time is a continuous non-decreasing process which increases only when $X_t\in\partial D$, namely for all $t\geq 0$ and q.e. $x\in\overline{D}$
\begin{equation*}
L_t=\int_0^t[\partial D](\X_s)\dd L_s,
\end{equation*}
$\mathbb{P}_x$-a.s. and
\begin{equation*}
\EW_x\int_0^t[\partial D](\X_s)\dd s=0. 
\end{equation*}
Although the reflecting diffusion process associated with 
(\ref{eqn:Dirichlet}) does in general not admit a Skorohod decomposition of the form (\ref{eqn:SkorohodRBM}), we may still define a continuous one-dimensional process with these properties. More precisely, by the Lipschitz property of $\partial D$, we have that $D\cap B(x,r_D)=\{(\tilde{x},x_d):x_d>\gamma(\tilde{x})\}\cap B(x,r_D)$ and the Lipschitz function $\gamma$ is differentiable a.e. with a bounded gradient. In particular, we have for every Borel set $B\subset \partial D\cap B(x,r_D)$ that
\begin{equation*}
\sigma(B)=\int_{\{\tilde{x}:(\tilde{x},\gamma(\tilde{x}))\in B\}}\Big(1+\lvert \nabla \gamma(\tilde{x})\rvert^2\Big)^{1/2}\dd\tilde{x}
\end{equation*}
and a straightforward computation yields that the Lebesgue surface measure $\sigma$ is a smooth measure with respect to $(\E,\D(\E))$ having finite energy, i.e., 
\begin{equation*}
\int_{\partial D}\lvert v\rvert \dd\sigma(x)\leq c\lvert\lvert v\rvert\rvert_{\E_1}\quad \text{for all }v\in \D(\E)\cap C(\overline{D}),
\end{equation*}
where we have used the inner product $\E_1(\cdot,\cdot):=\E(\cdot,\cdot)+\left<\cdot,\cdot\right>$.
\begin{definition}
The positive continuous additive functional of $X$ whose Revuz measure is given by the Lebesgue surface measure $\sigma$ on $\partial D$, i.e., the unique $L \in\mathcal{A}_c^+$ such that
\begin{equation}\label{eqn:Revuz}
\lim_{t\rightarrow 0+}\frac{1}{t}\int_D\EW_x\Big\{\int_0^t\phi(\X_s)\dd L_s\Big\}\psi(x)\dd x=
\int_{\partial D}\phi(x)\psi(x)\dd\sigma(x)
\end{equation}
for all non-negative Borel functions $\phi$ and all $\alpha$-excessive functions $\psi$, is called the \emph{boundary local time} of the reflecting diffusion process $\X$.
\end{definition}

The rest of this section is devoted to showing that the $\E$-exceptional set $\mathcal{N}$ is actually empty. Therefore, we consider 
the non-positive definite self-adjoint operator
$(\mathcal{L},\D(\mathcal{L}))$ associated with the Dirichlet form
$(\E,\D(\E))$. That is, for $v\in\D(\mathcal{L})$ we
have 
\begin{equation}\left<-\mathcal{L}v,w\right>=\E(v,w)\quad\text{ for all }w\in\D(\E)\end{equation} 
and 
the domain of $\mathcal{L}$ is given by
\begin{equation*}
\D(\mathcal{L})=\Big\{v\in\D(\E):\exists \phi\in L^2(D)\text{ s.t. }
\E(v,w)=\int_D\phi w\dd  x\ \forall w\in \D(\E)\Big\}.
\end{equation*}
In order to \emph{refine} the reflecting diffusion process $\X$ to start from every
$x\in\overline{D}$, we exploit the connection between the strongly continuous sub-Markovian contraction semigroup
$\{T_t,t\geq 0\}$ on $L^2(D)$ and the evolution system corresponding to
$(\mathcal{L},\mathcal{D}(\mathcal{L}))$, see, e.g., the monograph
\cite{Pazy}. Namely, for every $v_0\in L^2(D)$, the trajectory
$v:(0,T)\rightarrow H^1(D)$, $v(t)=T_tv_0$ belongs to the function space
$$\{\phi\in L^2((0,T);H^1(D)):\dot{\phi}\in
L^2((0,T);H^{-1}(D))\}$$ and is the unique mild solution to the parabolic \emph{abstract
Cauchy problem}
\begin{equation}\begin{aligned}\label{parabolic}
&\dot{v} + \mathcal{L} v=0 &\ &\text{in }(0,T)\\
&v(0)=v_0.&\ &
\end{aligned}\end{equation}
This is equivalent to the variational formulation 
\begin{equation}\label{weak}
-\int_0^T\left<v(t),w\right>\dot{\varphi}(t)\dd 
t+\int_0^T\left<\mathcal{L}v(t),w\right>\varphi(t)\dd 
t-\left<v_0,w\right>\varphi(0)=0
\end{equation}
for all $w\in H^1(D)$ and all $\varphi\in C^{\infty}_c([0,T))$.
Moreover, $T_t$ is known to be a bounded operator from $L^1(D)$
to $L^{\infty}(D)$ for every $t>0$. Therefore, by the Dunford-Pettis
theorem, it can be represented as an integral operator for every $t>0$,
\begin{equation}\label{eqn:kernel}
T_t\phi(x)=\int_D p(t,x,y)\phi(y)\dd  y\quad\text{ for
every } \phi\in L^1(D),
\end{equation}
where for all $t> 0$ we have $p(t,\cdot,\cdot)\in L^{\infty}(D\times D)$ and
$p(t,\cdot,\cdot)\geq 0$ a.e.
We call the function $p$ the \emph{transition kernel density}Êof $\X$.

The following proposition adapts a well-known result for diffusion processes on $\R^d$, cf. \cite{Stroock0}, which follows from the famous De Giorgi-Nash-Moser theorem, to the case of reflecting diffusion processes on $\overline{D}$. The key idea of the proof is the following \emph{extension by reflection} technique from \cite[Section 2.4.3]{Troianiello}:  We extend the solution to a parabolic problem by reflection at the boundary. Then we show that this extension again solves a parabolic problem so that we can apply the interior regularity result due to De Giorgi, Nash and Moser. See also the article \cite{Nittka} by Nittka, where such a technique is applied to elliptic boundary value problems.  

\begin{proposition}\label{thm:1}
$p\in C^{0,\delta}((0,T]\times\overline{D}\times\overline{D})$ for some $\delta\in(0,1)$, i.e., 
for each fixed $0<t_0\leq T$, there exists a positive constant $c$ such that 
\begin{equation}\label{eqn:Hold}
\lvert p(t_2,x_2,y_2)-p(t_1,x_1,y_1)\rvert\leq
c(\sqrt{t_2-t_1}+\lvert x_2-x_1\rvert+\lvert y_2-y_1\rvert)^{\delta}
\end{equation}
for all $t_0\leq t_1\leq t_2\leq T$ and all
  $(x_1,y_1),(x_2,y_2)\in\overline{D}\times\overline{D}$.
Moreover, the mapping $t\mapsto p(t,\cdot,\cdot)$ is analytic from $(0,\infty)$ to $C^{0,\delta}(\overline{D}\times\overline{D})$.  
\end{proposition}
\begin{proof}
First note that Nash's inequality holds for the underlying
Dirichlet form $(\E,H^1(D))$, i.e., there exists a constant $c_1>0$ such
that 
\begin{equation*}
\lvert\lvert v\rvert\rvert_2^{2+4/d} \leq c_1 (\E(v,v)+\lvert\lvert
v\rvert\rvert_2^2)\lvert\lvert v\rvert\rvert_{1}^{4/d}\quad\text{for
all }v\in H^1(D).
\end{equation*}
This is a direct consequence of the uniform ellipticity
(\ref{eqn:ellipticity}) and \cite[Corollary 2.2]{Bass}, where Nash's
inequality is shown to hold for the Dirichlet form $(\E^{\text{BM}},H^1(D))$ 
for a bounded Lipschitz domain $D$.
Analogously to the proof of \cite[Theorem 3.1]{Bass}, it follows thus
from \cite[Theorem 3.25]{Carlen} that the
transition kernel density satisfies an Aronson type Gaussian upper bound
\begin{equation}\label{eqn:Gaussian}
p(t,x,y)\leq c_1
t^{-d/2}\exp\Big(-\frac{\lvert x-y\rvert^2}{c_2t}\Big)
\end{equation}
for all $t\leq 1$ and all $(x,y)\in\overline{D}\times\overline{D}$. In particular, $\sup_{0< t\leq 1}\lvert\lvert
p(t,\cdot,\cdot)\rvert\rvert_{\infty}$ is finite and hence by the interior H\"older
continuity obtained from the De Giorgi-Nash-Moser theorem, cf.~\cite{Nash, Stroock0}, the estimate (\ref{eqn:Hold}) is true for
all $(x_1,y_1)$, $(x_2,y_2)$ satisfying $d(x_i,\partial
D),d(y_i,\partial D)>c_3$, $i=1,2$, for some constant $c_3>0$ and all $t_0\leq t_1\leq
t_2\leq 1$. Note that by the semigroup property the
Chapman-Kolmogorov equation holds, i.e.,
\begin{equation}\label{eqn:Chapman}
p(t_1+t_2,x,y)=\int_D p(t_1,x,z)\,p(t_2,z,y)\dd  z
\end{equation}
for every pair $t_1,t_2\geq 0$ and a.e. $x,y\in\overline{D}$. In particular, for fixed
$y\in\overline{D}$ the function $v:=p(\cdot,\cdot,y)$ is the unique
solution to (\ref{parabolic}) with initial value $v_0:=p(0,\cdot,y)\in
L^2(D)$. 
Now let $z\in\partial D$ so that by the Lipschitz property of $\partial D$ we
have after translation and rotation
$B(z,r_D)\cap\overline{D}=\{(\tilde{x},x_d)\in
B(z,r_D):x_d\geq \gamma(\tilde{x})\}$ and $B(z,r_D)\cap\partial
D=\{\tilde{x}\in B(z,r_D):x_d=\gamma(\tilde{x})\}$, where we have
introduced the notation $\tilde{x}=(x_1,...,x_{d-1})^T$. 
Let us furthermore introduce the one-to-one transformation
$\Psi(x):=(\tilde{x},x_d-\gamma(\tilde{x}))$ which straightens the boundary $B(z,r_D)\cap\partial
D$. $\Psi$ is a bi-Lipschitz transformation and the Jacobians of both
$\Psi$ and $\Psi^{-1}$ are bounded with bounds that depend only on the
Lipschitz constant $c_D$.
Since $v$ is the solution to (\ref{parabolic}) with appropriate initial
condition, the function $\hat{v}:=v(\cdot,\Psi^{-1}(\cdot))$ must
satisfy the variational formulation of the following parabolic problem in
$\hat{D}(z,r_D):=\Psi(B(z,r_D)\cap\overline{D})$
\begin{eqnarray*}
\int_0^T\dot{\varphi}(t)\int_{\hat{D}(z,r_D)}\hat{v}(t)w\dd  x\dd 
t&=&-\sum\limits_{i,j=1}^d\int_0^T\varphi(t)\int_{\hat{D}(z,r_D)}
\hat{\kappa}_{ij}\partial_i
\hat{v}(t)\partial_j w\dd x\dd  t\\
&&-\varphi(0)\int_{\hat{D}(z,r_D)}\hat{v}_0w\dd x
\end{eqnarray*}
for all $w\in C^{\infty}_c(\hat{D}(z,r_D))$ and all $\varphi\in
C^{\infty}_c([0,T))$. The coefficient $\hat{\kappa}$ is obtained via
change of variables and it is bounded and uniformly elliptic by the
boundedness of the Jacobians of $\Psi$ and $\Psi^{-1}$, respectively.
Now we use reflection at the hyperplane $\{(\tilde{y},0)\}$ via the
mapping $\rho(x):=(\tilde{x},-x_d)$ which yields that the function
$\hat{v}(\cdot,\rho(\cdot))$ satisfies the variational formulation of a parabolic problem on
$\rho(\hat{D}(z,r_D))$. Summing up both variational formulations on $
\hat{D}(z,r_D)$ and on $\rho( \hat{D}(z,r_D))$, respectively, we obtain
that the function 
\begin{equation*}
\check{v}(t,x):=\begin{cases}
\hat{v}(t,x),\quad&x\in \hat{D}(z,r_D)\\
\hat{v}(t,\rho(x)),\quad&x\in \rho(\hat{D}(z,r_D))
\end{cases}
\end{equation*}
satisfies the variational formulation of a parabolic problem in $\hat{D}(z,r_D)\cup\rho(\hat{D}(z,r_D)$.
By the interior H\"older estimate
for $\check{v}$, together with the fact that we may choose $c_3=r_D/4c_D$, we obtain thus
\begin{equation*}
\lvert p(t_2,x_2,\Psi^{-1}(y_2))-p(t_1,x_1,\Psi^{-1}(y_1))\rvert\leq
c_1(\sqrt{t_2-t_1}+\lvert y_2-y_1\rvert)^{c_2}
\end{equation*} 
for all $t_0\leq t_1\leq t_2\leq 1$ and $y_1,y_2\in
\{(\tilde{x},x_d):\lvert\tilde{x}\rvert<c_3,\ x_d\in
(0,r_D/4)\}$. As $\Psi$ is bi-Lipschitz, for fixed $x$, the mapping
$(t,y)\mapsto p(t,x,y)$ is H\"older continuous in
$(t_0,1]\times(B(z,c_3)\cap\overline{D})$ and by symmetry of the
transition kernel density the same holds true for the mapping
$(t,x)\mapsto p(t,x,y)$ for fixed $y$.  
Finally, the first assertion on
$(t_0,1]\times\overline{D}\times\overline{D}$ follows due to compactness of $\partial D$ 
and its generalization to arbitrary $T>0$ is obtained after
repeatedly applying the Chapman-Kolmogorov equation.  

The second assertion follows by the fact that the semigroup $\{T_t,t\geq 0\}$ extrapolates to a holomorphic semigroup on $L^2(D)$. 
More precisely, the semigroup possesses a holomorphic extension to the sector $\Sigma_{\theta}:=\{re^{i\alpha}:r>0,\lvert\alpha\rvert<\theta\}$ for some $\theta\in (0,\frac{\pi}{2}]$, cf., e.g., \cite{Pazy}. Let $0<t_0\leq T$ and set $$\Sigma_{\theta}(t_0,T):=\{z\in \mathbb{C}:z-t_0\in\Sigma_{\theta}, \lvert z\rvert<T\}.$$ By the H\"older continuity of $p$, the set $\{p(z,\cdot,\cdot):z\in\Sigma_{\theta}(t_0,T)\}$ is a bounded subset of $C^{0,\delta}(\overline{D}\times\overline{D})$. Moreover, the family of functionals obtained form integration against the functions $[B_1](x)[B_2](y)$ for measurable $B_1, B_2\subset \overline{D}$ form a separating subspace of $(C^{0,\delta}(\overline{D},\overline{D}))'$, i.e., for $k\in C^{0,\delta}(\overline{D}\times\overline{D})$ 
\begin{equation*}
\int_{D\times D}k(x,y)[B_1](x)[B_2](y)\dd x\dd y=0\quad\text{for all measurable }B_1, B_2\subset \overline{D}
\end{equation*}
implies that $k\equiv 0$. As the mapping 
$$z\mapsto\left< T_z[B_1],[B_2]\right>=\int_{D\times D}p(z,x,y)[B_1](y)[B_2](x)\dd x\dd y$$ is holomorphic for all $z\in\Sigma_{\theta}$, the mapping $z\mapsto p(z,\cdot,\cdot)$ is holomorphic from $\Sigma_{\theta}(t_0,T)$ to $C^{0,\delta}(\overline{D}\times\overline{D})$ by \cite[Theorem 3.1]{Arendt}. Since $t_0$ and $T$ were arbitrary, the assertion is proved.
\end{proof} 
By \cite[Theorem 2]{Fukushima2}, the existence of a H\"older continuous
transition kernel density ensures that we may refine the process $\X$ to
start from every $x\in\overline{D}$ by identifying the strongly
continuous semigroup $\{T_t,t\geq 0\}$ with the transition semigroup
$\{P_t,t\geq 0\}$. In particular, if $v$ is continuous and locally in $H^1(D)$, the Fukushima decomposition holds for every $x\in\overline{D}$, i.e., 
\begin{equation}\label{eqn:Fukudecomp}
v(\X_t)=v(\X_0)+M_t^{v}+N_t^{v},\quad\text{for all }
t> 0,
\end{equation} 
$\mathbb{P}_{x}$-a.s.,
where $M^{v}$ is a martingale additive functional of $\X$
having finite energy and $N^{v}$ is a continuous additive functional of $\X$ having
zero energy. 

Moreover, both $M^{v}$ and $N^{v}$ can be taken to be additive functionals of $\X$ in the strict
sense, cf. \cite[Theorem 5.2.5]{Fukushimaetal}. 

Finally, note that the $1$-potential of the Lebesgue surface measure $\sigma$ of $\partial D$ is the solution to an elliptic boundary value problem on a Lipschitz domain with bounded data. By elliptic regularity theory, cf., e.g.,  \cite{Griepentrog}, this solution is continuous, implying that the boundary local time $L$ exists as a positive continuous additive functional in the strict sense, cf.~\cite[Theorem 5.1.6]{Fukushimaetal}.

\section{Skorohod decompositions}\label{sectionSk}
In this section, we derive Skorohod decompositions of the reflecting diffusion process $\X$ for two practically relevant special cases, namely local Lipschitz conductivities and isotropic piecewise constant conductivities. 

The assertion of the following proposition is already covered by \cite[Theorem 2.3]{FukushimaTomisaki}; we include a proof for the sake of self-containedness.
\begin{proposition}\label{thm:Sk}
Let $\kappa\in C_{\text{loc}}^{0,1}(\overline{D};\R^{d\times d})$ be a symmetric, uniformly bounded and uniformly elliptic conductivity. Then the reflecting diffusion process $\X$
admits the following Skorohod decomposition
\begin{equation}\label{eqn:Skor1}
X_t = x+\int_0^t B(X_s)\dd W_s+\int_0^t\nabla\kappa(X_s)\dd s-\int_0^t\kappa(X_s)\nu(X_s)\dd L_s,
\end{equation}
$\mathbb{P}_x$-a.s., where $B:\overline{D}\rightarrow\R^{d\times d}$ denotes the positive definite \textit{diffusion matrix }satisfying $B^2=2\kappa$, $W$ is a standard $d$-dimensional Brownian motion and $L$ is the boundary local time of $\X$.
\end{proposition}

\begin{proof}
We have shown in Section \ref{section2}, that the Fukushima decomposition holds with a unique martingale additive functional $M^v$ in the strict sense and a unique continuous additive functional $N^v$ in the strict sense. Let us first compute the energy measure of $M^v$.
For $v,w\in\D(\E)$ we obtain 
\begin{eqnarray*}
\int_Dw(x)\dd\mu_{\left<M^{v}\right>}(x)&=&\lim_{t\rightarrow 0+}\frac{1}{t}\int_{D}\EW_{x}\left\{(v(X_t)-v(x_0))^2\right\}w(x)\dd x\\
&=&\lim_{t\rightarrow 0+}\frac{1}{t}\int_D(T_t v^2(x)-2v(x)T_tv(x)+v^2(x))w(x)\dd x\\
&=&\lim_{t\rightarrow 0+}\frac{2}{t}\int_{D}v(x)w(x)(v(x)-T_tv(x))\dd x\\
&& -\lim_{t\rightarrow 0+}\frac{1}{t}\int_Dv^2(x)(w(x)-T_tw(x))\dd x\\
&=& 2\E(vw,v)-\E(v^2,w)\\
&=&2\int_D\kappa\nabla v(x)\cdot\nabla v(x)\, w(x)\dd x,
\end{eqnarray*}
which yields the energy measure $$\dd \mu_{\left<M^{v}\right>}(x)=2\sum_{i,j=1}^d\kappa_{ij}(x)\partial_i v(x)\partial_jv(x)\dd x$$
so that the predictable quadratic variation of $M^v$ is given by
\begin{equation}\label{eqn:predqa}
\big<M^v\big>_t=2\int_0^t\sum_{i,j=1}^d\kappa_{ij}(\X_s)\partial_i v(\X_s)\partial_j v(\X_s)\dd s.
\end{equation}
Using the coordinate mappings $\phi_i(x):=x_i$, $i=1,...,d$, on $\overline{D}$
yields that $M^{\phi}$ is a continuous martingale additive functional in the strict sense with covariation
$$\big<M^{\phi_i},M^{\phi_j}\big>_t=2\int_0^t\kappa_{ij}(\X_s) \dd s,$$
$\mathbb{P}_x$-a.s. A standard characterization of continuous martingales, cf., e.g., \cite{Ikeda}, yields that  
\begin{equation}\label{eqn:diffusion}
M^{v}_t=\int_0^t (B(X_s)\nabla v(\X_s))\cdot\dd W_s,
\end{equation}
$\mathbb{P}_x$-a.s., where $B:\overline{D}\rightarrow\R^{d\times d}$ denotes the positive definite diffusion matrix satisfying $B^2=2\kappa$ and $W$ is a standard $d$-dimensional Brownian motion. 

Now let us consider the continuous additive functional $N^v$. 
Again using the coordinate mappings on $\overline{D}$, we obtain from Green's formula that 
\begin{eqnarray*}
\E(\phi_i,w)&=&\sum_{j=1}^d\int_D\kappa_{ij}(x)\partial_j w(x)\dd x\\
&=&-\sum_{j=1}^d\int_D\partial_j\kappa_{ij}(x)w(x)\dd x+\sum_{j=1}^d\int_{\partial D}\kappa_{ij}(x)\nu_j(x)w(x)\dd\sigma(x)
\end{eqnarray*}
for all $v\in H^1(D)\cap C(\overline{D})$. That is, $N^{\phi_i}$ is associated with the signed Radon measure 
\begin{equation*}
-\sum_{j=1}^d\partial_j\kappa_{ij}\dd x+\sum_{j=1}^d\kappa_{ij}(x)\nu_j(x)\dd\sigma(x)
\end{equation*}
and by the fact that the unique positive continuous additive functionals in the strict sense having as Revuz measure the Lebesgue measure and the Lebesgue surface measure, respectively, are given by the constant additive functional $t$ and the boundary local time $L_t$, we have shown that for every $x\in\overline{D}$
\begin{equation}\label{eqn:caf}
N_t^{\phi_i}=\sum_{j=1}^d\int_0^t\partial_j\kappa_{ij}(\X_s)\dd s-\sum_{j=1}^d\int_0^t\kappa_{ij}(\X_s)\nu_j(\X_s)\dd L_s,
\end{equation}
$\mathbb{P}_{x}$-a.s. Substitution of (\ref{eqn:diffusion}) and (\ref{eqn:caf}) in the Fukushima decomposition for the coordinate mappings finally yields the Skorohod decomposition (\ref{eqn:Skor1}).
\end{proof}

Now let us turn to the case of isotropic piecewise constant conductivities and for simplicity of the presentation let us consider a simplistic two-phase medium, where 
\begin{equation}\label{eqn:simplistic}
\kappa(x)=\begin{cases}
\kappa_1,\quad x\in D_1\\
\kappa_2,\quad x\in D_2,
\end{cases}
\end{equation}
with constants $\kappa_1,\kappa_2>0$ and $D$ is a simply connected bounded Lipschitz domain which consists of two disjoint subdomains such that $D_1=D\backslash\overline{D}_2$. We assume that $D_2$ is a simply connected Lipschitz domain. $\nu$ is the outer unit normal vector on $\partial D$ and the outer unit normal vector on $\partial D_2$ with respect to $D_2$.
\begin{definition}
The positive continuous additive functional $L^{0}$ of $X$ whose Revuz measure is given by the scaled Lebesgue surface measure $(\kappa_1+\kappa_2)\sigma$ on $\partial D_2$ is called the \emph{symmetric local time} of the reflecting diffusion process $\X$ at $\partial D_2$. 
\end{definition}
\begin{remark}
The term \lq\lq symmetric\rq\rq\ comes from the fact that in the one-dimensional case $L^{0}$ is the local time defined by the Tanaka formula with the convention $\operatorname{sign}(0)=0$, which is called the symmetric local time, see \cite{Revuz}. In this case we have
\begin{equation*}
L_t^0=\lim_{\ep\rightarrow 0}\frac{1}{2\ep}\int_0^t[[-\ep,\ep]](X_s)\dd s.
\end{equation*}
\end{remark}

\begin{proposition}
Let $\kappa$ be given by (\ref{eqn:simplistic}). Then the reflecting diffusion process $\X$
admits the following Skorohod decomposition
\begin{equation}\label{eqn:Skor1}
\X_t = x+\int_0^t\sqrt{2\kappa(\X_s)}\dd W_s+\frac{\kappa_1-\kappa_2}{\kappa_1+\kappa_2}\int_0^t\nu(\X_s)\dd L^{0}_s-\kappa_1\int_0^t\nu(\X_s)\dd L_s,
\end{equation}
 $\mathbb{P}_x$-a.s., where $W$ is a standard $d$-dimensional Brownian motion, $L^0$ is the symmetric local time of $X$ at $\partial D_2$ and $L$ is the boundary local time.
\end{proposition}
\begin{proof}
Repeating the computations from the proof of Proposition \ref{thm:Sk} yields for the martingale additive functional the predictable quadratic variation
 $$\big<M^{v}\big>_t=2\int_0^t\kappa(\X_s)\lvert\nabla v(X_s)\rvert^2 \dd s,\quad\mathbb{P}_x\text{-a.s. for every }x\in\overline{D},$$
implying that 
\begin{equation*}
M^{v}_t=\int_0^t \sqrt{2\kappa(\X_s)}\nabla v(\X_s))\cdot\dd W_s,\quad\mathbb{P}_x\text{-a.s. for every }x\in\overline{D},
\end{equation*}
where $W$ is a standard $d$-dimensional Brownian motion.
By Green's formula we have for all $w\in \D(\E)\cap C(\overline{D})$ 
\begin{eqnarray*}
\E(v,w)&=&\kappa_1\int_{D_1}\nabla v\cdot\nabla w\dd x+\kappa_2\int_{D_2}\nabla v\cdot\nabla w\dd x\\
&=&-\int_Dv\kappa\Delta v\dd x-(\kappa_1-\kappa_2)\int_{\partial D_2}\partial_{\nu}v w\dd\sigma(x)\\
&&+\kappa_1\int_{\partial D}\partial_{\nu}v w\dd\sigma(x).
\end{eqnarray*}
Using the coordinate mappings on $\overline{D}$ we obtain that $N^{\phi_i}$ is associated with the signed Radon measure 
\begin{equation*}
-(\kappa_1-\kappa_2)\nu(x)\dd\sigma\vert_{\partial D_2}(x)+\kappa_1\nu(x)\dd\sigma\vert_{\partial D}(x).
\end{equation*}
The assertion follows by the fact that the unique positive continuous additive functionals in the strict sense having as Revuz measure the (scaled) Lebesgue surface measure on $\partial D_2$ and $\partial D$, respectively, are given by the symmetric local time $L^0$ at $\partial D_2$ and the boundary local time $L_t$.
\end{proof}

%%%%%%%%%%%%%%%%%%%%%%%%%%%%%%%%%%%%%%%%%%%%%%%%%%%

\section{Feynman-Kac formulae}\label{section4}
In this section, we derive the Feynman-Kac formulae for both the continuum model and the complete electrode model. 
Afterwards, we will obtain, as a corollary, a Feynman-Kac formula for the mixed boundary value problem corresponding to the stochastic anomaly detection problem introduced in the first chapter.
Compared to the earlier works \cite{Hsu,Papanicolaou,Pardoux} on Feynman-Kac formulae, the main difficulty in deriving these formulae in our particular setting comes from the lack of It\^{o}'s formula for general reflecting diffusion processes.

The rest of this subsection is devoted to providing some auxiliary lemmata.
\begin{lemma}\label{lem:exponential}
The transition kernel density $p$ approaches the stationary distribution
uniformly and exponentially fast, that is, there exist positive
constants $c_1$ and $c_2$ such that for all
$(x,y)\in\overline{D}\times\overline{D}$ and every $t\geq 0$,
\begin{equation}\label{eqn:stationary}
\lvert p(t,x,y)-{\lvert D\rvert}^{-1}\rvert\leq c_1\exp(-c_2t).
\end{equation}
\end{lemma}
\begin{proof}
Fix $x\in\overline{D}$. Then by the Chapman-Kolmogorov equation 
\begin{eqnarray*}
p(t,x,x)-\lvert D\rvert^{-1}&=&\int_Dp(t/2,x,y)p(t/2,y,x)\dd y-\lvert D\rvert^{-1}\\
&=&\int_D(p(t/2,x,y))^2\dd y-\lvert D\rvert^{-1}\\
&=&\int_D(p(t/2,x,y)-\lvert D\rvert^{-1})^2\dd y\geq 0,
\end{eqnarray*}
where we have used that $\int_D p(t,x,y)\dd y = 1$ for all $t\geq 0$ and $x\in\overline{D}$.
Moreover, we have by the analyticity of the mapping $t\mapsto p(t,x,x)$, cf. Theorem \ref{thm:1}, and the fact that $p$ solves a parabolic boundary value problem with homogeneous Neumann boundary condition
\begin{equation*}
\frac{\mathrm{d}}{\mathrm{d}t}(p(t,x,x)-{\lvert D\rvert}^{-1})=\frac{\mathrm{d}}{\mathrm{d}t}\int_Dp(t/2,x,y)p(t/2,y,x)\dd y
\end{equation*}
\begin{eqnarray*}
\phantom{\frac{\mathrm{d}}{\mathrm{d}t}(p(t,x,x)-{\lvert D\rvert}^{-1})}
&=&\frac{\mathrm{d}}{\mathrm{d}t}\int_D(p(t/2,x,y))^2\dd y\\
&=&-\int_D\kappa(y) \nabla_yp(t/2,x,y)\cdot\nabla_y p(t/2,x,y)\dd y\\
&\leq&-c^{-1}\int_D\nabla_y p(t/2,x,y)\cdot\nabla_y p(t/2,x,y)\dd y\\
&\leq&-c^{-1}c_D^{-1}\Big(\int_D(p(t/2,x,y))^2\dd y-\lvert D\rvert^{-1}\Big)\\
&=&-c^{-1}c_D^{-1}\Big(p(t,x,x)-\lvert D\rvert^{-1}\Big),  
\end{eqnarray*}
where we have used (\ref{eqn:ellipticity}) and the Poincar\'e inequality
\begin{equation*}
\Big\lvert\Big\lvert \phi-\lvert D\rvert^{-1}\int_D\phi(x)\dd x\Big\rvert\Big\rvert_2\leq c_D\lvert\lvert \nabla\phi\rvert\rvert_2\quad\text{for all }\phi\in H^1(D).
\end{equation*}
Integration of the inequality from above
yields a diagonal estimate, i.e., there exist positive constants $c_1$ and $c_2$ such that 
\begin{equation*}
0\leq p(t,x,x)-\lvert D\rvert^{-1}\leq c_1\exp(-c_2t)\quad\text{for every }t\geq 0.
\end{equation*}
Now, the assertion follows from the Cauchy-Schwarz inequality and the fact that, by the computations from above, we may write the expression $\lvert p(t,x,y)-\lvert D\rvert^{-1}\rvert$ in the form 
\begin{eqnarray*}
&&\Big\lvert\int_D(p(t/2,x,z)-\lvert D\rvert^{-1})(p(t/2,z,y)-\lvert D\rvert^{-1})\dd z\Big\rvert\\
&\leq&\Big(\int_D(p(t/2,x,z)-\lvert D\rvert^{-1})^2\dd z\Big)^{1/2}\Big(\int_D(p(t/2,y,z)-\lvert D\rvert^{-1})^2\dd z\Big)^{1/2}\\
&\leq&\Big(p(t/2,x,x)-\lvert D\rvert^{-1}\Big)^{1/2}\Big(p(t/2,y,y)-\lvert D\rvert^{-1}\Big)^{1/2}
\end{eqnarray*}
\end{proof}

\begin{lemma}\label{lem:test}
Let $\kappa_{ij}\in C^{\infty}(\overline{D};\R^{d\times d})$, $1\leq i,j\leq d$. Then the set 
\begin{equation}\label{eqn:test}
V_{\kappa}(D):=\{\phi : \phi\in C^2(D),\partial_{\kappa\nu}\phi=0\text{ a.e. on }\partial D \}\cap H^1(D)
\end{equation}
is dense in $H^1(D)$. 
\end{lemma}
\begin{proof}
Diagonalizing the operator $(\mathcal{L},\D(\mathcal{L}))$ corresponding to the conductivity $\kappa$, we obtain an orthonormal basis
$\{\phi_k,k\in\N\}$ of $L^2(D)$ and an increasing sequence
$(\lambda_k)_{k\in\N}$ of real positive numbers which tend to infinity
such that for every $k\in\N$, $\phi_k\in H^1(D)$ is a weak solution of
the corresponding eigenvalue problem with homogeneous Neumann boundary condition. 
Note that the inner product $\E_1(\cdot,\cdot)$ is equivalent to the standard inner product on $H^1(D)$. Note that $V_{\kappa}(D)$ contains the linear span of
$\{\phi_k,k\in\N\}$ by elliptic regularity so it is enough to show that the
linear span of eigenfunctions is dense.
Therefore, let $\psi\in H^1(D)$ such that $\E_1(\phi_k,\psi)=0$ 
for every $k\in\N$, then 
\begin{equation*}
0=\int_D\kappa\nabla\phi_k\cdot\nabla\psi\dd  x
+\left<\phi_k,\psi\right>=(\lambda_k+1)\left<\phi_k,\psi\right>.
\end{equation*} 
Hence it follows $\left<\phi_k,\psi\right>=0$ for every $k\in\N$ and
the fact that $\{\phi_k,k\in\N\}$ is an orthonormal basis of $L^2(D)$ 
implies $\psi\equiv 0$ which proves the assertion.
\end{proof}

\begin{lemma}\label{lem:occ}
For every $x\in\overline{D}$ and every bounded Borel
function $\phi$ on $\partial D$ we have
 \begin{equation}\label{eqn:occ}
\EW_{x}\int_0^t \phi(\X_s)\dd  L_s=\int_0^t\int_{\partial
D}p(s,x,y)\phi(y)\dd \sigma(y)\dd  s\quad\text{for all }t\geq 0.
\end{equation}
\end{lemma}

\begin{proof}
First, the expression (\ref{eqn:occ}) is well-defined since the boundary local time of $\X$ exists as a
positive continuous additive functional in the strict sense. 
Without loss of generality we may assume that $\phi$ is non-negative. It
follows from \cite[Theorem 5.1.3]{Fukushimaetal} that the Revuz
correspondence (\ref{eqn:Revuz}) is equivalent to 
\begin{eqnarray*}
\int_D\psi(x)\,\EW_{x}\int_0^t\phi(\X_s)\dd  L_s\dd 
x&=&\int_0^t\int_{\partial D}\phi(y)T_s\psi(y)\dd \sigma(y)\dd  s\\
&=&\int_D \psi(x)\int_0^t\int_{\partial
D}p(s,y,x)\phi(y)\dd \sigma(y)\dd  s\dd x
\end{eqnarray*} 
for every $t>0$ and all non-negative Borel functions $\psi$ and $\phi$, where we have used Fubini's theorem in the second line. As this holds for every
non-negative Borel function $\psi$, we deduce
\begin{equation*}
\EW_{x}\int_0^t\phi(\X_s)\dd  L_s= \int_0^t\int_{\partial D}
p(s,x,y)\phi(y)\dd \sigma(y)\dd  s\quad \text{a.e. in }\overline{D}.
\end{equation*}
To obtain the assertion everywhere in $\overline{D}$, fix an arbitrary $x_0\in\overline{D}$ and consider for
$t_0>0$ the integral
\begin{eqnarray*}
\EW_{x_0}\int_{t_0}^t\phi(\X_s)\dd 
L_s&=&\int_Dp(t_0,x_0,x)\,\EW_{x}\Big\{\int_0^{t-t_0}\phi(\X_s)\dd  L_s\Big\}\dd  x\\
&=&\int_Dp(t_0,x_0,x)\Big(\int_0^{t-t_0}\int_{\partial D}p(s,x,y)\phi(y)\dd \sigma(y)\dd s\Big)\dd x\\
&=&\int_{t_0}^t\int_{\partial D} p(s,x_0,y)\phi(z)\dd \sigma(y)\dd 
s,
\end{eqnarray*}
where we have used the Markov property of $\X$.
Now let $(t_k)_{k\in\N}$ denote a positive sequence which monotonically decreases to zero as $k\rightarrow\infty$. By the computation from above we have for every $x\in\overline{D}$
\begin{equation*}
\EW_{x}\int_0^t\phi(\X_s)\dd L_s=\int_{t_k}^t\int_{\partial D}p(s,x,y)\phi(y)\dd\sigma(y)\dd s+\EW_{x}\int_0^{t_k}\phi(\X_s)\dd L_s
\end{equation*}
The claim follows by the facts that $\phi$ is bounded and $\EW_{x}L_{t_k}$ goes to zero as $k\rightarrow\infty$ which follows from monotonicity and continuity of the local time and the property $L_0=0$ $\mathbb{P}_x$-a.s. for every $x\in\overline{D}$.
\end{proof}

\subsection{Continuum model}\label{subsec:cont}
The main result for the continuum model (\ref{eqn:con}),
(\ref{eqn:continuum}) is the following theorem.
\begin{theorem}\label{thm:cont}
Let $f$ be a bounded Borel function satisfying $\left<f,1\right>_{\partial D}=0$.
Then there is a unique weak solution $u\in C(\overline{D})\cap H^1_{\diamond}(D)$ to the
boundary value problem (\ref{eqn:con}), (\ref{eqn:continuum}). This solution admits the
Feynman-Kac representation
\begin{equation}
u(x)=\lim_{t\rightarrow\infty}\EW_{x}\int_0^t f(\X_s)\dd 
L_s\quad\text{for all }x\in\overline{D}.
\end{equation}  
\end{theorem}

\begin{proof} 
The existence of a unique normalized weak solution $u$ to (\ref{eqn:con}),
(\ref{eqn:continuum}) is guaranteed by the standard theory of linear
elliptic boundary value problems. 
Let us set $$u_t(x):=\EW_{x}\int_0^t f(\X_s)\dd  L_s\quad\text{and}\quad
u_{\infty}(x):=\lim_{t\rightarrow\infty}u_t(x),\quad x\in\overline{D},$$
respectively.
From the occupation formula (\ref{eqn:occ}) and the compatibility
condition $(\ref{eqn:compatibility})$, it follows immediately
that 
\begin{equation*}
u_t(x)=\int_0^t\int_{\partial D}(p(s,x,y)-\lvert
D\rvert^{-1})f(y)\dd \sigma(y)\dd  s\quad\text{for all }x\in\overline{D}.
\end{equation*}
By Lemma \ref{lem:exponential} the convergence towards the stationary
distribution is uniform over $\overline{D}$, in particular, 
\begin{equation}\label{eqn:occ1}
u_{\infty}(x)=\int_0^{\infty}\int_{\partial D}(p(s,x,y)-\lvert
D\rvert^{-1})f(y)\dd \sigma(y)\dd  s\quad\text{for all
}x\in\overline{D}. 
\end{equation}

It follows from (\ref{eqn:occ1}) together with the H\"older continuity shown in Proposition \ref{thm:1} and
the Aronson type upper bound (\ref{eqn:Gaussian}) that $u_{\infty}$ is
in $C(\overline{D})$. Moreover, by the facts that $p$ is the transition kernel density of a reflecting diffusion process and $f$ is bounded, Fubini's theorem yields
\begin{equation*}
\int_D\Big(\int_0^{\infty}\int_{\partial D}p(s,x,y)f(y)\dd\sigma(y)\dd s\Big)\dd x=\int_0^{\infty}\int_{\partial D}f(y)\dd\sigma(y)\dd s=0,
\end{equation*}
i.e., $u_{\infty}\in H^1_{\diamond}(D)$.

Now, let us use the following regularization technique in order to show $u\equiv u_{\infty}$: Let
$(\kappa^{(k)})_{k\in\N}$ denote a sequence of smooth conductivities
with components in $C^{\infty}(\overline{D})$ such that for $1\leq
i,j\leq d$, $\kappa_{ij}^{(k)}\rightarrow\kappa_{ij}$ a.e. as
$k\rightarrow\infty$. Let us consider the Dirichlet forms
$(\E^{(k)},H^1(D))$, $k\in\N$, with $$\E^{(k)}(v,w):=\int_D\kappa^{(k)}\nabla
v\cdot\nabla w\dd  x$$ and the associated reflecting diffusion processes
$\X^{(k)}$. By Proposition \ref{thm:Sk}, we obtain the Skorohod decomposition
\begin{equation*}
\X_t^{(k)}=x+\int_0^ta^{(k)}(\X_s^{(k)})\dd  s+\int_0^t B^{(k)}(\X_s)\dd 
W_s-\int_0^t\kappa^{(k)}(\X_s^{(k)})\nu(\X^{(k)}_s)\dd  L^{(k)}_s,
\end{equation*}  
where $W$ is a standard $d$-dimensional Brownian motion,
$a^{(k)}_i:=\sum_{j=1}^d\partial_j\kappa_{ij}^{(k)}$, $i=1,...,d,$ and
the matrix $B^{(k)}$ satisfies $2\kappa^{(k)}=(B^{(k)})^2$.
Let us define $u^{(k)}_t$
in the same manner as $u_t$ and
$u^{(k)}(x):=\lim_{t\rightarrow\infty}u_{t}^{(k)}(x)$,
$x\in\overline{D}$. We show that $u^{(k)}$ is the unique weak solution
of the elliptic boundary value problem
\begin{equation*}\begin{cases}
\nabla\cdot(\kappa^{(k)}\nabla u^{(k)})=0\quad&\text{in }D \\
\partial_{\kappa^{(k)}\nu}u^{(k)}=f\quad&\text{on }\partial D
\end{cases}
\end{equation*}
in the Sobolev space $H^1_{\diamond}(D).$ For test functions
$v\in V_{\kappa^{(k)}}(D)$, we may apply It\^{o}'s formula for semimartingales to obtain 
\begin{equation*}
\EW_xv(\X^{(k)}_t)=v(x)+\EW_x\int_0^t
\nabla\cdot(\kappa^{(k)}\nabla v(\X_s^{(k)}))\dd  s. 
\end{equation*}
By Fubini's theorem, this is equivalent to 
\begin{equation*}
T_t^{(k)}v(x)-v(x)=\int_0^t\int_Dp^{(k)}(s,x,y)\nabla\cdot(\kappa^{(k)}\nabla
v(y))\dd  y\dd  s,
\end{equation*}
where we have used the superscript \lq\lq$(k)$\rq\rq\ for the semigroup and
transition kernel density, respectively, corresponding to
$\kappa^{(k)}$. Multiplication with $f$, integration over $\partial D$
and another change of the orders of integration finally yield 
\begin{equation*}
\int_{\partial D}f(y)(T_t^{(k)}v(y)-v(y))\dd \sigma(y)
=\left<u_t^{(k)},\nabla\cdot(\kappa^{(k)}\nabla v)\right>.
\end{equation*}
Since $u_t^{(k)}\rightarrow u^{(k)}$ and $T^{(k)}_tv\rightarrow\lvert
D\rvert^{-1}\int_D v\dd x$, both uniformly on $\overline{D}$, as
$t\rightarrow\infty$, we have 
\begin{equation*} 
\left<u^{(k)},\nabla\cdot(\kappa^{(k)}\nabla v)\right>
=-\left<f,v\right>_{\partial D},
\end{equation*}
where we have used the expression (\ref{eqn:occ1}) with $p^{(k)}$
instead of $p$ for $u^{(k)}$. As this holds true for every $v\in V_{\kappa^{(k)}}(D)$,
$u^{(k)}$ must be the unique normalized weak solution to the boundary
value problem by a density argument. 
 
Next, we show the convergence of the sequence $(u^{(k)})_{k\in\N}$ as $k\rightarrow\infty$
towards $u\in H^1_{\diamond}(D)$, 
the unique solution to
(\ref{eqn:con}), (\ref{eqn:continuum}).
From our assumptions on the sequence $(\kappa^{(k)})_{k\in\N}$, it is clear that $\kappa_{ij}^{(k)}\rightarrow \kappa_{ij}$, $1\leq i,j\leq d$, in $L^2(D)$ as $k\rightarrow 0$, which implies $G$-convergence of the sequence of elliptic operators $(\mathcal{L}^{(k)})_{k\in\N}$ on $D$ towards $\mathcal{L}$, cf. \cite{Zikov1}.
That is, for any $\phi\in (H^1_0(D))'$, the solutions $w^{(k)}\in H^1_0(D)$ to the homogeneous Dirichlet problem 
\begin{equation}\label{hom:Dir}
\nabla\cdot(\kappa^{(k)}\nabla w^{(k)})=\phi\quad\text{in }D
\end{equation}
satisfy $w^{(k)}\rightharpoonup w$ in $H^1_0(D)$ as $k\rightarrow\infty$ and $\kappa^{(k)}\nabla w^{(k)}\rightharpoonup \kappa\nabla w$ in $L^2(D;\R^d)$ as $k\rightarrow\infty$, where $w\in H^1_0(D)$ is the solution to the homogeneous Dirichlet problem
\begin{equation*}
\nabla\cdot(\kappa\nabla w)=\phi\quad\text{in }D.
\end{equation*}
Consider the variational form of the Neumann problem for $u^{(k)}\in H^1_{\diamond}(D)$, i.e., 
\begin{equation*}
\int_D\kappa^{(k)}\nabla u^{(k)}\cdot\nabla v\dd x=\left<f,v\right>_{\partial D}\quad\text{for all }v\in H^1_{\diamond}(D).
\end{equation*}
As $u^{(k)}\in H^1_{\diamond}(D)$ for all $k\in\N$, we have by the Poincar\'{e} inequality
\begin{equation*}
\lvert\lvert u^{(k)}\rvert\rvert\leq c_1\lvert\lvert \nabla u^{(k)}\rvert\rvert_{2}\leq c_2\lvert\lvert f\rvert\rvert_{2}.
\end{equation*}
That is, the sequence $(u^{(k)})_{k\in\N}$ is bounded in $H^1_{\diamond}(D)$ and we may extract a weakly convergent subsequence, still denoted $(u^{(k)})_{k\in\N}$ for convenience. We have $u^{(k)}\rightharpoonup \tilde{u}$ in $H^1_{\diamond}(D)$, as $k\rightarrow\infty$, so that it remains to show the convergence of the flows
\begin{equation*}\kappa^{(k)}\nabla u^{(k)}\rightharpoonup \kappa \nabla \tilde{u}\quad\text{in }L^2(D), \text{ as }k\rightarrow\infty.\end{equation*} 
We choose an arbitrary $w\in C_0^{\infty}(D)$, set $\phi:=\nabla\cdot(\kappa\nabla w)$ and consider the solutions $w^{(k)}$, $k\in\N$, of the corresponding homogeneous Dirichlet problem (\ref{hom:Dir}). Then we have for every $k\in\N$ the trivial identity
\begin{equation*}
\int_D\kappa^{(k)}\nabla w^{(k)}\cdot \nabla u^{(k)}\dd x=\int_D\kappa^{(k)}\nabla u^{(k)}\cdot \nabla w^{(k)}\dd x
\end{equation*}
and the convergence of the flows follows from the compensated compactness lemma \cite[Lemma 1.1]{Zikov1}.
We may thus pass to the limit in the variational formulation to see that $\tilde{u}\equiv u$ is the unique solution to the Neumann problem. 

On the other hand, by \cite[Lemma 2.2]{RoSl00} together with the H\"older continuity up to the boundary of both $p^{(k)}$, $k\in\N$, and $p$, it follows that for fixed $x\in\overline{D}$,  $p^{(k)}(\cdot,x,\cdot)\rightarrow p(\cdot,x,\cdot)$ uniformly on compacts in $(0,T]\times\overline{D}$ for all $T>0$. It follows from (\ref{eqn:occ1}) that $u^{(k)}(x)\rightarrow u_{\infty}(x)$ for all $x\in\overline{D}$ as $k\rightarrow\infty$. Hence, $u$ must coincide with $u_{\infty}$ and the assertion is proved.
\end{proof}
\begin{remark}\label{rem:1}
Note that the regularization technique we employed in the proof of Theorem \ref{thm:cont} may be easily modified to prove the Feynman-Kac formula 
\begin{equation*}
u(x)=\EW_x\phi(\X_{\tau(D)}),\quad x\in D
\end{equation*}
for the conductivity equation (\ref{eqn:con}) with Dirichlet boundary condition $u\vert_{\partial D}=\phi,$ where $\phi\in H^{1/2}(D)$ and $$\tau(D):=\inf\{t\geq 0:\X_t\in\R^d\backslash D\}$$ denotes the \emph{first exit time} from the domain $D$. Such a proof requires the fact that
\begin{equation*}
\X^{(k)}_{\cdot\wedge \tau^{(k)}(D)}\rightarrow \X_{\cdot\wedge\tau(D)}\quad\text{in law on }C([0,\infty);\overline{D})
\end{equation*}
as $n\rightarrow\infty$ for every $x\in D$.
This follows immediately from the Lipschitz property of $\partial D$, implying that all points of $\partial D$ are \emph{regular} in the sense of \cite[Chapter 4.2]{Karatzas}, see also the proof of Theorem \ref{thm:hom} in the next chapter.
\end{remark}
A slight modification of the arguments from above yields the following result which is in fact a corollary rather to the proof of Theorem \ref{thm:cont} than to its actual statement. 
\begin{corollary}\label{cor:continuum}
Let $f$ be a bounded Borel function and let $\alpha>0$.
Then there is a unique weak solution $u\in C(\overline{D})\cap H^1_{\diamond}(D)$ to the
boundary value problem 
\begin{equation*}
\nabla\cdot(\kappa\nabla u)-\alpha u=0\quad\text{in D}
\end{equation*}
subject to the boundary condition (\ref{eqn:continuum}). This solution admits the
Feynman-Kac representation
\begin{equation}
u(x)=\EW_{x}\int_0^{\infty} e^{-\alpha t}f(\X_t)\dd 
L_t\quad\text{for all }x\in\overline{D}.
\end{equation}  
\end{corollary}
\begin{proof}
Repeat the proof of Theorem \ref{thm:cont}, however, substituting $\{T_t,t\geq 0\}$ with the \emph{Feynman-Kac semigroup} $\{\widetilde{T}_t,t\geq 0\}$, $\widetilde{T}_tv(x):=\EW_xe^{-\alpha t}v(\X_t)$. Note that in contrast to the Neumann problem without the zero-order term, the \emph{gauge function}
$\EW_x \int_0^{\infty}e^{-t}\dd L_t$
is finite $\mathbb{P}_x$-a.s. for every $x\in\overline{D}$.
\end{proof}

\subsection{Complete electrode model}\label{sub:cem}
The main result for the complete electrode model (\ref{eqn:con}),
(\ref{eqn:cem}) is the following theorem.
\begin{theorem}\label{thm:cem}
For given functions $f,g$ defined by (\ref{eqn:functions}) and a voltage
pattern $U\in\R^N$ satisfying (\ref{eqn:ground}), there is a unique
weak solution $u\in C(\overline{D})\cap H^1(D)$ to the boundary value problem
(\ref{eqn:con}), (\ref{eqn:cem}). This solution admits the Feynman-Kac
representation 
\begin{equation}\label{eqn:FK}
u(x)=\EW_{x}\int_0^{\infty}e_g(t)f(\X_t)\dd 
L_t\quad\text{for all }x\in\overline{D},
\end{equation}
with 
\begin{equation}
e_g(t):=\exp\Big(-\int_0^tg(\X_s)\dd  L_s\Big),\quad t\geq 0.
\end{equation}
\end{theorem}
Before we are ready to give a proof of Theorem \ref{thm:cem}, let us
introduce the \emph{Feynman-Kac semigroup} of the complete electrode
model, i.e., the one-parameter family of operators $\{T^g_t,t\geq 0\}$
defined by
\begin{equation} 
T^g_tv(x):=\EW_x e_g(t)v(\X_t),\quad x\in\overline{D}\text{ and } t\geq 0.
\end{equation}
Let us define the \emph{perturbed} Dirichlet form $(\E^g,\D(\E^g)))$ by a perturbation of $(\E,\D(\E))$ with the measure $g\cdot\sigma$, i.e.,
\begin{equation}
\E^g(v,w):=\E(v,w)+\left<gv,w\right>_{\partial D},\quad v,w\in\D(\E^g),
\end{equation}
where $\D(\E^g)=H^1(D)$ by the standard trace theorem.

\begin{proposition}\label{thm:3}
The Feynman-Kac semigroup $\{T^g_t,t\geq 0\}$ is a strong Feller semigroup on $L^2(D)$.
\end{proposition}
\begin{proof}
First, it follows from a straightforward computation that $\{T^g_t,t\geq 0\}$  is associated with the 
Dirichlet form $(\E^g,\D(\E^g))$ and is therefore a strongly continuous sub-Markovian contraction semigroup on
$L^2(D)$. 
$T^g_t$ is a bounded operator from $L^1(D)$ to $L^{\infty}(D)$ for every $t>0$, which can be
shown using Fatou's lemma. By the Dunford-Pettis theorem, $T^g$ can thus
be represented as an integral operator for every $t>0$,
\begin{equation}
T^g_t\phi(x)=\int_Dp^g(t,x,y)\phi(y)\dd  y,\text{
for every }\phi\in L^1(D), 
\end{equation}
where for all $t>0$ we have $p^g(t,x,y)\in L^{\infty}(D\times D)$ and
$p^g(t,x,y)\geq 0$ for a.e. $x, y \in \overline{D}$.  In order to prove the strong Feller
property, we show
that $T^g_t$, $t>0$ maps bounded Borel functions to
$C(\overline{D})$. As in the papers \cite{Hsu3,Papanicolaou}, we use an iterative method
 to construct the transition kernel density
$p^g$: Let $p^g_0(t,x,y):=p(t,x,y)$ and set 
\begin{equation*}
p_k^g(t,x,y):=-\int_0^t\int_{\partial
D}p(s,x,z)g(z)p_{k-1}^g(t-s,z,y)\dd \sigma(z)\dd  s,\quad k\in\N.
\end{equation*}
Note that the terms $p_k^g$ are symmetric in the $x$ and
$y$ variables by the symmetry of $p$. By induction using Lemma
\ref{lem:occ} it is not difficult to verify that for all $k\in\N$
\begin{equation*}
\int_0^t\int_{\partial D}g(x)p_k^g(s,x,y)\dd \sigma(x)\dd  s\leq
\Big(\sup_{x\in\overline{D}}\Big\{\EW_x\int_0^tg(\X_s)\dd 
L_s\Big\}\Big)^{k+1}
\end{equation*}
and that there is a positive constant $c_1$ such that 
\begin{equation}\label{eqn:densest}
\lvert p_k^g(t,x,y)\rvert\leq c_1^{k+1}t^{-d/2}\Big(\sup_{x\in\overline{D}}
\Big\{\EW_x\int_0^tg(\X_s)\dd  L_s\Big\}\Big)^{k+1}\quad
\text{for all }k\in\N.
\end{equation}
Let us show the continuity of $p^g_k$, $k\in\N\cup\{0\}$. For $k=0$, this
is a consequence of Theorem \ref{thm:1}. Now assume that $p^g_{k-1}$ is
continuous on $(t_0,T]\times\overline{D}\times\overline{D}$ for $t_0>0$,
then we have for $t\in (t_0,T]$
\begin{equation*}\begin{split}
p_k^g(t,x,y)=-&\int_0^{t_0}\int_{\partial D}p(s,x,z)g(z)
p_{k-1}^g(t-s,z,y)\dd \sigma(z)\dd  s\\
&-\int_{t_0}^{t}\int_{\partial
D}p(s,x,z)g(z)p_{k-1}^g(t-s,z,y)\dd \sigma(z)\dd  s.
\end{split}
\end{equation*}
Note that the first integral on the right-hand side tends to zero
uniformly as $t_0\rightarrow 0$, which is a consequence of
(\ref{eqn:densest}), while the second integral is continuous by
assumption. Hence, there exists a $T>0$ such that the series
$p^g(t,x,y):=\sum_{k=0}^{\infty}p_k^g(t,x,y)$ converges absolutely and
uniformly on any compact subset of
$(0,T]\times\overline{D}\times\overline{D}$ and is thus continuous on
$(0,T]\times\overline{D}\times\overline{D}$. By the Markov property we
have for all $t\in (0,T]$ and every $x\in\overline{D}$ the following 
expression for $T^g_t\phi(x)$
\begin{equation*}
\int_Dp^g(t,x,y)\phi(y)\dd  y=\EW_x\phi(\X_t)+\sum_{k=1}^{\infty}
\frac{1}{k!}\EW_x\Big\{\Big(-\int_0^tg(\X_s)\dd 
L_s\Big)^k\phi(\X_t)\Big\}.
\end{equation*}
Therefore, the assertion for arbitrary $T>0$ follows from the Chapman-Kolmo-gorov
equation.
\end{proof}
\begin{corollary}\label{cor:cont}
Let $u$ be defined by the Feynman-Kac formula (\ref{eqn:FK}), then $u\in C(\overline{D})$.
\end{corollary}
\begin{proof}
Let us define a $\mathbb{P}_x$-martingale by 
\begin{equation*}
\EW_x\Big\{\int_0^{\infty}e_g(s)f(\X_s)\dd  L_s\vert\F_t\Big\}
=\int_0^te_g(s)f(\X_s)\dd  L_s+e_g(t)u(\X_t),
\end{equation*}
where the right-hand side is obtained using the Markov property of $\X$
together with the fact that $e_g$ is a multiplicative functional of
$\X$. Obviously,
\begin{equation*}
e_g(t)u(\X_t)-u(x)+\int_0^te_g(s)f(\X_s)\dd  L_s
\end{equation*} 
is a $\mathbb{P}_x$-martingale as well, and hence we
have for all $0\leq s\leq t$
\begin{equation*}
e_g(s)u(\X_s)=e_g(s)\EW_{\X_s}e_g(t-s)u(\X_{t-s})+e_g(s)
\EW_{\X_s}\int_0^{t-s}e_g(r)f(\X_r)\dd  L_r.
\end{equation*}
Setting $s=0$ thus yields  
\begin{equation}\label{eqn:marteq}
u(x)=T^g_t u(x)+\EW_x\int_0^t e_g(r)f(\X_r)\dd  L_r
\quad\text{for all }t\geq 0. 
\end{equation}
$T^g_tu$ is continuous on $\overline{D}$ by Proposition \ref{thm:3}. To
prove that $u$ is continuous on $\overline{D}$, it is sufficient to show
that the second term on the right-hand side of (\ref{eqn:marteq}) tends to zero uniformly in
$x$, as $t\rightarrow 0$. This is, however, clear since we may estimate
\begin{equation*}
\sup_{x\in\overline{D}}\Big\{\EW_x\int_0^te_g(s)f(\X_s) \dd 
L_s\Big\}\leq
z^{-1}\max_{l=1,...,N}\{U_l\}\sup_{x\in\overline{D}}\{\EW_xL_t\},
\end{equation*}
where the right-hand side tends to zero as $t\rightarrow 0$ by Lemma
\ref{lem:occ}.
\end{proof}

The following lemma yields a semimartingale decomposition for the composite process $u(\X_t)$ which compensates for the lack of It\^{o}'s formula in the proof of Theorem \ref{thm:cem}.
\begin{lemma}\label{lem:sem}
Let $u\in H^1(D)$ denote the weak solution of the boundary value problem (\ref{eqn:con}), (\ref{eqn:cem}). Then for all $t\geq 0$
\begin{equation}
u(\X_t)=u(x)+\int_0^t\nabla u(\X_s)\dd M^u_s-\int_0^tf(\X_s)\dd L_s+\int_0^tg(\X_s) u(\X_s)\dd L_s,
\end{equation}
$\mathbb{P}_x\text{-a.s.}$ for q.e. $x\in\overline{D}$.
\end{lemma} 
\begin{proof}
Applying the Fukushima decomposition (\ref{eqn:Fukudecomp}) to the perturbed Dirichlet form $(\E^g,H^1(D))$, we obtain the unique decomposition
\begin{equation*}
v(\X^g_t)-v(\X_0^g)=M_t^{g,v}+N_t^{g,v},\quad\text{for all }t>0,\quad\mathbb{P}_x\text{-a.s. for q.e. }x\in\overline{D}
\end{equation*}
into a martingale additive functional of finite energy and a continuous additive functional of zero energy of the non-conservative Hunt process $\X^g$ associated with $(\E^g,H^1(D))$. We study the relation between the continuous additive functionals $N^v$ and $N^{g,v}$. Let us assume first that $v$ is in the range of the $1$-resolvent associated with the perturbed Dirichlet form, i.e., 
$$v(x)=G_1^g\phi(x)=\EW_x\Big\{\int_0^{\infty}e^{-t-\int_0^tg(\X_s)\dd L_s}\phi(\X_t)\dd t\Big\},\quad x\in\overline{D}$$ 
for some $\phi\in L^2(D)$. Then we have by the resolvent property, $-\mathcal{L}G_1^g\phi=\phi-v$ so that for all $w\in H^1(D)$
\begin{equation*}
\E^g(G_1^g\phi,w)=\int_D(\phi-v)w\dd x.
\end{equation*}
Using the Revuz correspondence, we see that $N^{g,v}$ admits a semimartingale decomposition, namely $$N^{g,v}_t=\int_0^t(\phi(\X^g_s)-v(\X^g_s))\dd s.$$ 
Moreover, an easy computation yields that
\begin{equation}\label{eqn:alphapot}
G_{1}\phi-G_{1}^g\phi=\EW_x\Big\{\int_0^{\infty}e^{- t}g(\X_t)G_{1}^g\phi(\X_t)\dd L_t \Big\}\quad\text{for all }\alpha>0.
\end{equation}
By Corollary \ref{cor:continuum}, the right-hand side in (\ref{eqn:alphapot}) is the unique weak solution of the elliptic boundary value problem
\begin{equation*}
\begin{cases}
\nabla\cdot(\kappa\nabla w)- w=0&\quad\text{in }D\\
\partial_{\kappa\nu}w=gG_{1}^g\phi&\quad\text{on }\partial D.
\end{cases}
\end{equation*}
That is, it coincides with the $1$-potential $U_{1}(G_{1}^g\phi(g\cdot\sigma))\in H^1(D)$. 
Invoking the Revuz correspondence once more, we see that the zero energy continuous additive functional in the Fukushima decomposition of the $1$-potential corresponding to the signed Radon measure $vg\cdot\sigma$ is given by 
\begin{equation*}
N_t^{U_{1}(vg\cdot\sigma)}=\int_0^t{U_1(vg\cdot\sigma)}(\X_s)\dd s -\int_{0}^{t}v(\X_s)g(\X_s)\dd L_s,\quad\text{for all }t\geq 0.
\end{equation*}
Therefore, we obtain
\begin{eqnarray*}
N_t^v&=&N_t^{G_1\phi}-N_t^{U_1(vg\cdot\sigma)}\\
&=&\int_0^t(\phi(\X_s)-G_1\phi(\X_s))\dd s -\int_0^tU_1(vg\cdot\sigma)(\X_s)\dd s+\int_{0}^{t}v(\X_s)g(\X_s)\dd L_s\\
&=&\int_0^t(\phi(\X_s)-v(\X_s))\dd s+\int_0^tv(\X_s)g(\X_s)\dd L_s.
\end{eqnarray*}
Moreover, notice that $\X^g$ is related to $\X$ by a random time change, namely
\begin{equation*}
\X_s^g=\begin{cases}
\X_s,\quad &s<\zeta^g\\
\partial,\quad &s\geq\zeta^g,
\end{cases}
\end{equation*}
where the lifetime $\zeta^g$ is given by 
$$\zeta^g:=\inf\Big\{t:\int_0^tg(\X_s)\dd L_s>Z\Big\}$$ and $Z$ is an exponentially distributed random variable with parameter $1$.
Hence, we obtain
\begin{equation}\label{eqn:zero}
N_t^v=N_t^{g,v}+\int_0^tv(\X_s^g)g(\X_s^g)\dd L_s\quad\text{for all }t<\zeta^g.
\end{equation}
This equality may be generalized to the case of an arbitrary $v\in H^1(D)$ not necessarily in the range of the resolvent using an approximation argument. Namely, we consider the sequence $(v_k)_{k\in\N}$ with $v_k:=kG^g_{k+1}v=G^k_1\phi_k$, $\phi_k:=k(v-kG^g_{k+1}v)$. Then $v_k\in H^1(D)$ for all $k\in\N$ and the sequence $(v_k)_{k\in\N}$ satisfies both,
\begin{equation*} 
\lim_{k\rightarrow\infty}\E^g(v_k-v,v_k-v)=0\quad\text{ and }\quad\lim_{k\rightarrow\infty}\E(v_k-v,v_k-v)=0
\end{equation*}
so that by \cite[Corollary 5.2.1]{Fukushimaetal}, there exists a subsequence, for convenience still denoted $(v_k)_{k\in\N}$, such that $v_k(\X_t^g)\rightarrow v(\X_t^g)$, $N_t^{g,v_k}\rightarrow N_t^{g,v}$ and $N_t^{v_k}\rightarrow N_t^v$ uniformly on each finite time interval, $\mathbb{P}_x$-a.s. for q.e. $x\in\overline{D}$. In particular, it follows that (\ref{eqn:zero}) holds for arbitrary $v\in H^1(D)$.

As $u$ solves the boundary value problem (\ref{eqn:con}), (\ref{eqn:cem}) we have that 
\begin{equation*}
\E^g(u,v)=\left<f,v\right>_{\partial D}\quad\text{for all }v\in H^1(D)\cap C(\overline{D})
\end{equation*}
Since the perturbed Dirichlet form $(\E^g,H^1(D))$ is regular, 
we obtain from the Revuz correspondence the representation
\begin{equation*}
N^{g,u}_t=-\int_0^{t}f(\X^g_s)\dd L_s,\quad\mathbb{P}_x\text{-a.s. for q.e. }x\in\overline{D}.
\end{equation*} 
Finally, using the representation 
 \begin{equation*}
 M_t^u=\int_0^t\nabla u(\X_s)\dd M_s^u,\quad\mathbb{P}_x\text{-a.s. for q.e. }x\in\overline{D}
 \end{equation*}
 for the martingale additive functional, the claim follows from the Fukushima decomposition (\ref{eqn:Fukudecomp}), (\ref{eqn:zero}) and the Markov property of $\X$.\end{proof}

\begin{proof}[Proof of Theorem \ref{thm:cem}]
There exists a weak solution $u\in H^1(D)$ of the boundary
value problem (\ref{eqn:con}), (\ref{eqn:cem}) so that with regard to Corollary \ref{cor:cont}, it remains to show that this weak solution $u$ admits the Feynman-Kac representation (\ref{eqn:FK}). Note
first that the gauge function 
\begin{equation}\label{eqn:gauge}
\EW_x\int_0^{\infty}e_g(t)\dd  L_t
\end{equation}
is finite $\mathbb{P}_x$-a.s. for every $x\in\overline{D}$, hence the
expression on the right-hand side of (\ref{eqn:FK}) is well-defined.
Lemma \ref{lem:sem} yields the semimartingale decomposition
\begin{equation*}
u(\X_t)=u(x)+\int_0^t\nabla u(\X_s)\dd 
M_s^u-\int_0^t f(\X_s)\dd  L_s+\int_0^t g(\X_s) u(\X_s)\dd 
L_s,
\end{equation*}
$\mathbb{P}_x$-a.s. for q.e. $x\in\overline{D}$. Note that the second term on the right-hand side is a local $\mathbb{P}_x$-martingale and that $e_g$ is continuous, adapted
to $\{\F_t,t\geq 0\}$ and of bounded variation. Multiplication by such functions leaves
the class of local martingales invariant.
Using integration by parts we thus obtain for q.e.
$x\in\overline{D}$ the identity
\begin{equation*}
u(\X_t)e_g(t)=u(x)+\int_0^te_g(s)\nabla
u(\X_s)\dd  M_s^u-\int_0^t e_g(s)f(\X_s)\dd  L_s,
\end{equation*}
$\mathbb{P}_x$-a.s., where the second summand on the right-hand side is a local $\mathbb{P}_x$-martingale.
That is, there exists an increasing sequence $(\tau_k)_{k\in\N}$ of
stopping times which tend to infinity such that for every $k\in\N$ 
$$\mathcal{M}_{t\wedge\tau_k}:=\int_0^{t\wedge\tau_k}e_g(s)\nabla u(\X_s)\dd  M_s^u$$
is a $\mathbb{P}_x$-martingale. By definition of the term $e_g$, it is, however, clear that 
\begin{equation*}
\EW_x\sup_{k\in\N}\lvert \mathcal{M}_{t\wedge\tau_k}\rvert<\infty\quad\text{for all }t\geq 0\text{ and every }x\in\overline{D}
\end{equation*} 
which is sufficient for $\{\mathcal{M}_t,t\geq 0\}$ to be a $\mathbb{P}_x$-martingale by the dominated convergence theorem.
Hence,
\begin{equation*}
u(x)=\EW_x\int_0^te_g(s)f(\X_s)\dd  L_s+
\EW_x u(\X_t)e_g(t)\quad\text{for q.e.  }x\in\overline{D}.
\end{equation*}
Letting $t\rightarrow\infty$ finally yields 
\begin{equation*}
u(x)=\EW_x\int_0^{\infty}e_g(t)f(\X_t)\dd 
L_t\quad\text{for q.e. }x\in\overline{D},
\end{equation*}
where we have used the fact that $u$ is essentially bounded by standard elliptic regularity theory. Finally, by the fact that we have actually
shown in Corollary \ref{cor:cont} that the right-hand side in the last equality is continuous up to
the boundary, the assertion holds for every
$x\in\overline{D}$.
\end{proof}
\begin{remark}
Note that the technique we used to prove Theorem \ref{thm:cem} fails for
the Neumann problem corresponding to the continuum model. This comes
from the fact that in this case the gauge function (\ref{eqn:gauge}) becomes infinite.
For the same reason Theorem 1.2 from \cite{ChenZhang}, specialized to a
zero lower-order term, does not yield the desired Feynman-Kac
formula for the continuum model either.
\end{remark}

\subsection{Mixed boundary value problems}
Now we can directly deduce the desired Feynman-Kac formula for the mixed boundary value problem corresponding to the stochastic anomaly detection problem introduced in Section \ref{sec:stochan} of the previous chapter. Recall that in this setting $\partial D$ consists of two disjoint parts $\partial_1 D$ and $\partial_2 D$ and that measurements can be taken only on the \emph{accessible boundary} $\partial_1 D$ while the electric potential vanishes on the \emph{inaccessible boundary}Ê$\partial_2 D$. The deterministic EIT forward problem for the complete electrode model is then given by the conductivity equation (\ref{eqn:con}) subject to the mixed boundary conditions
\begin{equation}\label{eqn:cem3}
\begin{aligned}
&\kappa\nu\cdot\nabla u\vert_{\partial D}+g u\vert_{\partial D}=f\quad&\text{on }\partial_1 D\phantom{.}\\
&u\vert_{\partial D}=0\quad&\text{on }\partial_2 D.
\end{aligned}
\end{equation}  
The following result is a corollary to the line of arguments that led to the proof of Theorem \ref{thm:cem} rather than to its actual statement.  
\begin{corollary}\label{cor:cem}
For given functions $f,g$ defined by (\ref{eqn:functions}) and a voltage
pattern $U\in\R^N$ satisfying (\ref{eqn:ground}), there is a unique
weak solution $u\in C(\overline{D})\cap H^1(D)$ to the boundary value problem
(\ref{eqn:con}), (\ref{eqn:cem3}). This solution admits the Feynman-Kac representation 
\begin{equation}\label{eqn:FK_mixed}
u(x)=\EW_{x}\int_0^{\tau}e_g(t)f(\X_t)\dd 
L_t\quad\text{for all }x\in\overline{D},
\end{equation}
where $\tau:=\inf\{t\geq 0:\X_t\in\partial_2 D\}$ denotes the first hitting time of $\partial_2 D$.
\end{corollary}
\begin{proof}
Repeat the computations from Subsection \ref{sub:cem} with the Feynman-Kac semigroup $\{\widetilde{T}_t^g,t\geq 0\}$, where $\widetilde{T}_t^g v(x):=\EW_x\{[t\leq \tau]e_g(t)v(\X_t)\}$ instead of $\{T_t^g,t\geq 0\}$.
\end{proof}

%%%%%%%%%%%%%%%%%%%%%%%%%%%%%%%%%%%%%%%%%%%%%%%%%%%

\section{Stochastic homogenization}\label{section5}
In this section, we show that the stochastic EIT forward problem may be homogenized both theoretically and numerically by homogenization of the underlying diffusion process on the whole space $\R^d$. Moreover, we provide a continuum version of a quantitative estimate which has been obtained recently for the discrete random walk in random environment in \cite{Gloria,Egloffe}.
\subsection{Preliminaries}
For convenience of the reader, let us recall some standard concepts from homogenization theory. 
Let $\phi:=(\phi_1,...,\phi_d)$, $\phi_i\in L^2_{\text{loc}}(\R^d)$, $i=1,...,d$, denote a vector field. We say that $\phi$ is a \emph{gradient field} if for every $\psi\in C_c^{\infty}(\R^d)$,
\begin{equation*} 
\int_{\R^d}\phi_i\partial_j\psi -\phi_j\partial_i\psi\dd x=0,\quad i,j=1,...,d.
\end{equation*}
Moreover, we say that $\phi$ is \emph{divergence-free} if 
for every $\psi\in C_c^{\infty}(\R^d)$,
\begin{equation*} 
\sum_{i=1}^d\int_{\R^d}\phi_i\partial_i\psi \dd x=0.
\end{equation*}

Now let us consider a conductivity random field $\{\kappa(x,\omega),(x,\omega)\in \R^d\times\Gamma\}$ and let
 $\{\boldsymbol{\Theta}_x, x\in\R^d\}$ denote the underlying dynamical system which is assumed to satisfy the assumptions (i)-(iv) from Subsection \ref{subsec:proba}. A vector field  $\boldsymbol{\phi}\in L^2(\Gamma;\R^d)$ is called a \emph{gradient field}, respectively \emph{divergence-free}, if its realizations 
 $\phi(\cdot,\omega): \R^d\mapsto \R^d$, $x\mapsto \boldsymbol{\phi}(\boldsymbol{\Theta}_x\omega)$ are gradient fields, respectively divergence-free, for $\mathcal{P}$-a.e. $\omega\in\Gamma$. We define the function spaces \begin{eqnarray*}
&L^2_{\text{pot}}(\Gamma):=\{\boldsymbol{\phi}\in L^2(\Gamma;\R^d):\phi(\cdot,\omega)\text{ is a gradient field }\mathcal{P}\text{-a.s.}\}\\
&L^2_{\text{sol}}(\Gamma):=\{\boldsymbol{\phi}\in L^2(\Gamma;\R^d):\phi(\cdot,\omega)\text{ is divergence-free }\mathcal{P}\text{-a.s.}\}. 
\end{eqnarray*}
If $\boldsymbol{\phi}\in L^2_{\text{pot}}(\Gamma)$, we can find a function $\eta:\R^d\times\Gamma\rightarrow\R$ such that $\eta(\cdot,\omega)\in H^1_{\text{loc}}(\R^d)$ and 
\begin{equation}\label{eqn:eta}
\nabla\eta(\cdot,\omega)=\boldsymbol{\phi}(\boldsymbol{\Theta}_{\cdot}\omega)\quad\text{a.e. in }\R^d\text{ for }\mathcal{P}\text{-a.e. }\omega\in\Gamma.
\end{equation}
In particular, (\ref{eqn:eta}) defines a stationary random field with respect to the measure $\mathcal{P}$. We call $\eta$ the \emph{potential} corresponding to $\boldsymbol{\phi}$. 
\begin{remark}
Note that $\boldsymbol{\phi}\in L^2_{\text{pot}}(\Gamma)$ does not imply that $\{\phi(x,\omega),(x,\omega)\in\R^d\times\Gamma\}$ is a stationary random field with respect to $\mathcal{P}$. In fact, it can be shown that this is not true for $d=1$.  
\end{remark}
\noindent We define another function space
\begin{equation*}
\mathcal{V}^2_{\text{pot}}:=\{\boldsymbol{\phi}\in L^2_{\text{pot}}(\Gamma):\mathbb{M}\boldsymbol{\phi}=0\},
\end{equation*} 
so that one obtains an orthogonal Weyl decomposition of $ L^2(\Gamma;\R^d)$, namely
\begin{equation*} 
 L^2(\Gamma;\R^d)=\mathcal{V}^2_{\text{pot}}(\Gamma)\oplus L^2_{\text{sol}}(\Gamma),
 \end{equation*}
 cf., e.g., \cite{Zikov1}.
Let $\xi\in\R^d$ denote a direction vector, i.e., $\lvert \xi\rvert = 1$. The so-called \emph{auxiliary problem }Êfor the direction $\xi$ reads as follows: Find $\boldsymbol{\chi}^{\xi}\in \mathcal{V}^2_{\text{pot}}(\Gamma)$ such that $\boldsymbol{\kappa}(\xi+\boldsymbol{\chi}^{\xi})\in L^2_{\text{sol}}(\Gamma)$ or equivalently,
\begin{equation}\label{eqn:auxiliary}
\mathbb{M}\{\boldsymbol{\kappa}(\xi+\boldsymbol{\chi}^{\xi})\cdot\boldsymbol{\phi}\}=0\quad\text{for all }\boldsymbol{\phi}\in\mathcal{V}_{\text{pot}}^2(\Gamma).
\end{equation} 
For a proof of existence and uniqueness of the solution to the auxiliary problem we refer the reader to the seminal paper \cite{Pa1} by Papanicolaou and Varadhan.

We can now bring the underlying diffusion processes evolving in the random medium into play by recalling a stochastic homogenization result which was obtained by Lejay \cite{Lejay0}. Let $\{\kappa_{\ep}(x,\omega),(x,\omega)\in\R^d\times\Gamma\}$ denote the scaled conductivity random field, see Subsection \ref{subsec:proba}, and let $\X^{\omega,\ep}$ denote the diffusion process on $\R^d$ which is associated with the regular symmetric Dirichlet form 
\begin{equation*}\E^{\omega,\ep}(v,w):=\int_{\R^d}\kappa_{\ep}(\cdot,\omega)\nabla v\cdot\nabla w\dd x, \quad v,w\in \D(\E^{\omega,\ep}):=H^1(\R^d)
\end{equation*} 
on $L^2(\R^d)$. It has been shown in \cite{Lejay0} that, under assumption (A1) from Subsection \ref{subsec:proba}, for $\mathcal{P}$-a.e. $\omega\in\Gamma$
\begin{equation}\label{eqn:qual}
X^{\omega,\ep}_{\cdot}\rightarrow X_{\cdot}^*\quad\text{in law on }C([0,\infty);\R^d)\text{ as }\ep\rightarrow 0,
\end{equation}
where $X^*$ denotes the diffusion process on $\R^d$ which is associated with the \emph{homogenized Dirichlet form}
\begin{equation*}
\E^{*}(v,w):=\int_{\R^d}\kappa^*\nabla v\cdot\nabla w\dd x, \quad v,w\in \D(\E^{*}):=H^1(\R^d)
\end{equation*} 
on $L^2(\R^d)$ and the constant, symmetric and positive definite matrix $\kappa^*$ satisfies the equation
\begin{equation}\label{eqn:const_kappa}
\xi\cdot\kappa^*\xi=\mathbb{M}\{(\xi+\boldsymbol{\chi}^{\xi})\cdot\boldsymbol{\kappa}(\xi+\boldsymbol{\chi}^{\xi})\},
\end{equation}
where $\boldsymbol{\chi}^{\xi}\in \mathcal{V}^2_{\text{pot}}(\Gamma)$ denotes the solution to the auxiliary problem (\ref{eqn:auxiliary}) for the direction $\xi\in\R^d$. 

\subsection{Homogenization of the EIT forward problem}\label{sec:hom2}
The following theorem is our main result. Its assertion is in fact a rather direct consequence of an invariance principle for reflecting diffusion processes obtained recently by Chen, Croydon and Kumagai \cite{Chen1} and the Feynman-Kac formula (\ref{eqn:FK_mixed}) from Corollary \ref{cor:cem}.
\pagebreak
\begin{theorem}\label{thm:hom}
Let $\{\kappa_{\ep}(x,\omega),(x,\omega)\in\R^d\times\Gamma\}$ be a stationary random field satisfying assumption (A1) from Subsection \ref{subsec:proba} and assume that the trajectories satisfy $\kappa(\cdot,\omega)\in C_{\text{loc}}^{0,1}(\overline{D};\R^{d\times d})$ or $\kappa(\cdot,\omega)$ piecewise constant for $\mathcal{P}$-a.e. $\omega\in\Gamma$; let $\Sigma=\emptyset$. Then, for a given voltage pattern $U\in\R^N$, we have for the potentials in the stochastic boundary value problem (\ref{eqn:random_cond}), (\ref{eqn:fp2})
\begin{equation}\label{eqn:pw}
u_{\ep}(x,\omega)\rightarrow u^*(x), \quad x\in\overline{D}\text{ for }\mathcal{P}\text{-a.e. }\omega\in\Gamma,\text{ as }\ep\rightarrow 0,
\end{equation}
and the corresponding electrode currents satisfy 
\begin{equation}\label{eqn:hom}
\lim\limits_{\ep\rightarrow 0}J_l(\ep,\omega)=\frac{1}{\lvert E_l\rvert}\int_{E_l}\kappa^*\nu\cdot\nabla u^{*}(x)\vert_{\partial_1 D}\dd\sigma(x) \quad\text{for }\mathcal{P}\text{-a.e. }\omega\in\Gamma,
\end{equation}
$l=1,...,N$, where the function $u^*\in H^1_{0}(D\cup\partial_1 D)\cap C(\overline{D})$ is the unique solution to the deterministic forward problem
\begin{equation}\label{eqn:homcon}
\nabla\cdot(\kappa^*\nabla u^*)=0\quad\text{in }D
\end{equation}
subject to the boundary conditions
\begin{equation}\label{eqn:fp3}
\begin{aligned}
&\kappa^*\nu\cdot\nabla u^*\vert_{\partial_1 D}+gu^*\vert_{\partial_1 D}=f\quad&\text{on }\partial_1 D\\
&u^*\vert_{\partial_2 D}=0\quad&\text{on }\partial_2 D\\
\end{aligned}
\end{equation}
with a constant, symmetric and positive definite matrix $\kappa^*$ given by (\ref{eqn:const_kappa}). 
\end{theorem}
\begin{proof}
Let us first show that for $\mathcal{P}$-a.e. $\omega\in\Gamma$, $u_{\ep}(x,\omega)\rightarrow u^*(x)$, $x\in\overline{D}$, as $\ep\rightarrow 0$. 
Let $(\ep_{k})_{k\in\mathbb{N}}$ be an arbitrary monotone decreasing null sequence and let $X^{\omega,\ep}$ denote the reflecting diffusion process on the half-space corresponding to the regular symmetric Dirichlet form $(\E^{\omega,\ep},H^1(\R^d_-\cup\R^{d-1}))$ on $L^2(\R^d_-\cup\R^{d-1})$.
By assumption (A1) from Subsection \ref{subsec:proba}, we deduce from \cite[Section 4]{Chen1} that for $\mathcal{P}$-a.e. $\omega\in\Gamma$
\begin{equation*}
X^{\omega,\ep_k}\rightarrow X^*\quad\text{in law on }C([0,\infty);\R^d_-\cup\R^{d-1}),\text{ as }k\rightarrow\infty,
\end{equation*}
where $X^*$ is the reflecting diffusion process on the half-space associated with the homogenized regular symmetric Dirichlet form $(\E^{*},H^1(\R^d_-\cup\R^{d-1}))$ on $L^2(\R^d_-\cup\R^{d-1})$. The constant, symmetric and positive definite matrix $\kappa^*$ is given by (\ref{eqn:const_kappa}).

Let us first show that for every $x\in\overline{D}$ and $\mathcal{P}$-a.e. realization $\omega\in\Gamma$ of the random medium
$$\X^{\omega,\ep_k}_{\cdot\wedge\tau^{\omega,\ep_k}}\rightarrow \X_{\cdot\wedge\tau^*}^*\quad\text{in law on }C([0,\infty);\R^d_-\cup\R^{d-1}),\text{ as }k\rightarrow\infty.$$ 
Consider the functional $F: C([0,\infty);\R^d_-\cup\R^{d-1})\rightarrow[0,\infty]$, $$\phi\mapsto\begin{cases}
\infty,\quad&\text{if for all }t\geq 0: \lvert \phi(t)\rvert <R,\\
\inf\{t\geq 0: \lvert \phi(t)\rvert=R\}\quad&\text{else},
\end{cases} 
$$ 
defined in such a way that $F(\X^{\omega,\ep_k})=\tau^{\omega,\ep_k}$. Let $(\phi_k)_{k\in\N}$ denote a sequence of continuous functions that converges uniformly towards $\phi$ on compacts in $[0,\infty)$. If $\lim\inf_{k\rightarrow\infty}F(\phi_k)$ is finite, then we may extract a subsequence, still denoted $(\phi_k)_{k\in\N}$ for convenience, such that $F(\phi_k)\rightarrow \lim\inf_{k\rightarrow\infty}F(\phi_k)$. We have
\begin{equation*}
\begin{split}
\lvert \phi(\lim\inf_{k\rightarrow\infty}F(\phi_k))-\phi_k(F(\phi_k))\rvert \leq&\lvert \phi(\lim\inf_{k\rightarrow\infty}F(\phi_k))-\phi(F(\phi_k))\rvert\\
&+\lvert \phi(F(\phi_k))-\phi_k(F(\phi_k))\rvert
\end{split}
\end{equation*}
and by our assumption, the right-hand side vanishes as
$k\rightarrow\infty$. From the closedness of $\partial_1 D$, we conclude
hence that $\phi(\lim\inf_{k\rightarrow\infty}F(\phi_k))\in \partial_1
D$. In particular we have that $F(\phi)\leq
\lim\inf_{k\rightarrow\infty}F(\phi_k)$ if $\sup_{t\in[0,T]}\lvert
\phi(t)-\phi_k(t)\rvert \rightarrow 0$, as $t\rightarrow\infty$, for all
$T>0$, i.e., $F$ is lower semi-continuous. Now assume that $\phi$ is a
discontinuity point of $F$, i.e., there is $\delta > 0$ and $k_0$ such
that $F(\phi) + \delta \le F(\phi_k)<\infty$ for all $k \ge k_0$.
Then it follows that $\phi(t)\in \overline{D}$ for all $t\in [F(\phi),\lim\inf_{k\rightarrow\infty}F(\phi_k))$. 
However, the boundary $\partial_1 D$ is regular in the sense of \cite[Chapter 4.2]{Karatzas}, that is, a diffusion process originating from $x\in\partial_1 D$ will immediately exit from $\overline{D}$ $\mathbb{P}_x^{\omega,\ep_k}$-a.s., respectively $\mathbb{P}_x^*$-a.s. In other words, the set of discontinuities of $F$ is a null set with respect to these measures and hence the claim follows from the continuous mapping theorem, cf.~\cite{Billingsley}.
Now, as in the proof of \cite[Theorem 5.1]{Rozkosz}, it follows with the Fukushima decompositions in Section \ref{section5} of the second chapter that
\begin{equation*}
(X^{\omega,\ep_k}_{\cdot\wedge\tau^{\omega,\ep_k}},L^{\omega,\ep_k}_{\cdot\wedge\tau^{\omega,\ep_k}})\rightarrow (X^*_{\cdot\wedge\tau^*},L^*_{\cdot\wedge\tau^*})\quad\text{in law on }C([0,\infty);\R^d_-\cup\R^{d-1}\times\R_+),
\end{equation*}
 as $k\rightarrow\infty$. With regard to the Feynman-Kac formula (\ref{eqn:FK_mixed}), the assertion (\ref{eqn:pw}) can be proved as follows:  

Suppose first that $f$ is continuous. Then we have to show that
\[
\limsup_{j \to \infty} \Big\lvert\EWÄ_x \int_0^{\tau_j} e_g^{\omega,\ep_j}(t) f(\X_t^{\omega,\ep_j})
\dd L_t^{\omega,\ep_j} -
\EW_x \int_0^{\tau} e_g(t) f(\X_t)
\dd L_t\Big\rvert = 0.
\]
We estimate this with a difference of truncated Riemann sums
\begin{equation*}
\begin{split}
\Big\lvert\sum_{k = 0}^{\lfloor N/h \rfloor} \Big(&\EW_x \int_{t_k}^{t_{k+1}}
e_g^{\omega,\ep_j}(t_k) f(\X_{t_k}^{\omega,\ep_j})[t_{k+1} < \tau_j]
\dd L_t^{\omega,\ep_j}\\ 
&-\EW_x\int_{t_k}^{t_{k+1}}e_g(t_k) f(\X_{t_k})[t_{k+1} < \tau] \dd L_t
\Big)\Big\rvert,
\end{split}
\end{equation*}
where $t_k = kh$ and $h$ is a step size.
The difference $S_{N,h}$ of truncated Riemann sums for fixed $N$ and $h$ goes to
zero as $j \to \infty$ if we assume that $(\X^{\omega,\ep_j},
L^{\omega,\ep_j}, A^{\omega,\ep_j})$ converge weakly to $(\X, L, A)$, where
$A^{\omega, \ep_j} = \log e_g^{\omega, \ep_j}$ and $A = \log e_g$.

The error terms
for $\X^{\omega,\ep_j}$ and $\X$ are nearly analogous, so it is enough to consider
just $\X^{\omega,\ep_j}$ in detail. The increments (excluding the edge
case where $t_k < \tau_j < t_{k+1}$ which we will omit but which can be
treated in the same way) are of following form
\[
\begin{split}
&\int_{t_k}^{t_{k+1}} \big(e^{\omega,\ep_j}_g(t)f(\X^{\omega,\ep_j}_t) -
e^{\omega,\ep_j}_g(t_k)f(\X^{\omega,\ep_j}_{t_k})\big) \dd
L^{\omega,\ep_j}_t\\
&= \int_{t_k}^{t_{k+1}} \big(e^{\omega,\ep_j}_g(t)-
e^{\omega,\ep_j}_g(t_k)\big)f(\X^{\omega,\ep_j}_t) \dd L^{\omega,\ep_j}_t 
\\
&\phantom{\int_{t_k}^{t_{k+1}} \big(e^{\omega,\ep_j}_g(t)-
e^{\omega,\ep_j}_g(t_k)\big)}+e^{\omega,\ep_j}_g(t_k)\int_{t_k}^{t_{k+1}} \Big(f(\X^{\omega,\ep_j}_t) -
f(\X^{\omega,\ep_j}_{t_k})\Big) \dd L^{\omega,\ep_j}_t
\end{split}
\]
The latter term can be handled with the continuity of the paths of
$\X$ together with the uniform continuity of $f$ on the compact set
$\overline D$. This is since
\[
\begin{split}
&\Big\lvert e^{\omega,\ep_j}_g(t_k)\int_{t_k}^{t_{k+1}} \big(f(\X^{\omega,\ep_j}_t) -
f(\X^{\omega,\ep_j}_{t_k})\big) \dd L^{\omega,\ep_j}_t\Big\rvert  \\
&\le  \Big(2\big(1-\psi_{\delta}(\theta_h(\X^{\omega,\ep_j}))\big) \lvert\lvert f\rvert\rvert_\infty +
\theta_{2\delta}(f)
\psi_{\delta}(\theta_h(\X^{\omega,\ep_j}))
\Big)(L^{\omega,\ep_j}_{t_{k+1}} - L^{\omega,\ep_j}_{t_k}\big), 
\end{split}
\]
where $\theta_\delta(x)$ is the maximum variation of the function $x$ on
the interval $[0,N]$
\[
\theta_\delta(x): = \sup\big\{ \lvert x(t)-x(s)\rvert  \,; \, 0 \le t,s \le N, \lvert t-s\rvert <
\delta\big\}
\]
and $\psi_{\delta}(t)$ is the continuous approximation 
of the indicator function $[ t < \delta]$ with support in
$[-2\delta, 2\delta]$. Therefore, after taking the limit $j \to \infty$,
the latter terms give that the corresponding total approximation error can be bounded by
\[
4\|f\|_\infty \EW_x \{[\theta_h(\X) \ge \delta] L_{\tau\wedge N}\}
+ 2\theta_{2\delta}(f)\EW_x \{L_{\tau\wedge N}\}
\]
which goes to zero as $\delta \to 0$.

The first term can be estimated by
\[
\begin{split}
&\int_{t_k}^{t_{k+1}} \big(e^{\omega,\ep_j}_g(t)-
e^{\omega,\ep_j}_g(t_k)\big)f(\X^{\omega,\ep_j}_t) \dd L^{\omega,\ep_j}_t 
 \\
&\le 2 \|f\|_\infty \|g\|_\infty
\int_{t_k}^{t_{k+1}} \big(L^{\omega,\ep_j}_t - L^{\omega,\ep_j}_{t_k}\big)
\dd L^{\omega,\ep_j}_t=\|f\|_\infty \|g\|_\infty(L^{\omega,\ep_j}_{t_{k+1}} -
L^{\omega,\ep_j}_{t_k}\big)^2
\end{split}
\]
since $g \ge 0$. Therefore, after taking $j \to \infty$ the first term
gives the total approximation error that is bounded by
\[
2\|f\|_\infty\|g\|_\infty \Big(\EW_x \{[\theta_h( L) \ge \delta] L_{\tau\wedge N}\}
+ 2\delta\EW_x \{L_{\tau\wedge N}\}\Big)
\]
which goes to zero as $\delta \to 0$ as well.

We are still left with the truncation.  The truncation can be removed
since $\tau_j$ and $\tau$ are a.s. finite and moreover,
$L^{\omega,\ep_j}_{\tau_j}$ converges weakly to $L_\tau$ and thus, we get
a uniform estimate
\[
\begin{split}
&\limsup_{j \to \infty} \Big\lvert\EW_x \int_0^{\tau_j} e_g^{\omega,\ep_j}(t) f(\X_t^{\omega,\ep_j})
\dd L_t^{\omega,\ep_j} -
\int_0^{\tau} e_g(t) f(\X_t)
\dd L_t\Big\rvert\\ &\le 2 \|f\|_\infty \EW_x\{L_{\tau} - L_{\tau\wedge N}\}
\end{split}
\]
which gives the claimed convergence for continuous $f$ if we can prove
weak convergence of the sequence $(A^{\omega,\ep_j})_{j\in\N}$.

Therefore, we will next verify the weak convergence of $A^{\omega,\ep_j}$ jointly
with $(X^{\omega,\ep_j}, L^{\omega,\ep_j})$. We can assume that
$X^{\omega,\ep_j} \to X$ and $L^{\omega,\ep_j} \to L$ almost surely in
$C(0,T)$. Suppose that $g = \sum_{l=1}^N g_l [E_l]$ where $g_l \ge 0$
and continuous, $E_l \cap E_k = \emptyset$ and $\sigma(\partial E_l) =
0$. Therefore, it is enough to show the convergence for $g = [E_l]$ for
$E_l$ open in $\partial D$. Since $E_l$ is open, it can be approximated
from below by an increasing sequence of continuous functions $g'_k$ that
converge pointwise to $g$.

It follows from the almost sure convergence of $X^{\omega,\ep_j}$ and
$L^{\omega,\ep_j}$ that
\[
\liminf_{j \to \infty} \int_0^t g(\X^{\omega,\ep_j}_t) \dd L_t^{\omega,\ep_j} \ge \limsup_{k \to
\infty}\int_0^t g'_k(\X_t) \dd L_t = 
\int_0^t g(\X_t) \dd L_t.
\]
Since $\X^{\omega,\ep_j}$ and $\X$ will not hit $\partial E_l$ in
$[0,T]$ almost surely, we can get the other direction by considering
$1-g$. Therefore,
\[
\limsup_{j \to \infty} \int_0^t g(\X^{\omega,\ep_j}_t) \dd L_t^{\omega,\ep_j} \le 
\int_0^t g(\X_t) \dd L_t.
\]
Using countability of the rational numbers, we obtain
\[
\forall t \in [0,T]\cap \mathbb Q\colon \lim_{j \to \infty}
A^{\omega,\ep_j}(t) = A(t)
\]
almost surely and by monotonicity this implies the almost sure
convergence of $A^{\omega,\ep_j}$ to $A$ in $C(0,T)$, which implies the
weak convergence of the original versions. The same technique allows us to extend the assertion to the case of discontinuous $f$.

For the proof of (\ref{eqn:hom}) note that the boundary condition (\ref{eqn:fp2}) allows us to write 
\begin{equation}\label{eqn:currentep}
J_l(\ep_k,\omega)=\frac{1}{\lvert E_l\rvert}\int_{E_l}(f-gu_{\ep_k}(\cdot,\omega)\vert_{\partial_1 D})\dd\sigma(x)\quad \text{for }\mathcal{P}\text{-a.e. } \omega\in\Gamma 
\end{equation}
and that (\ref{eqn:fp2}) may be written in the form $(\Lambda_{\kappa_{\ep}}+gI)u_{\ep}=f$. We deduce from the well-posedness of the forward problem that $(\Lambda_{\kappa_{\ep}}+gI)^{-1}$, $L^2(\partial D)\rightarrow L^2(\partial D)$, is bounded. 
Since, due to assumption (A1), the corresponding constant does not depend on $\ep$, the sequence $(u_{\ep_k})_{k\in\N}$ is bounded in $L^2(\partial D)$ for $\mathcal{P}$-a.e. $\omega\in\Gamma$ implying its uniform integrability, see, e.g., \cite{Klenke}. As we already know the pointwise convergence $u_{\ep_k}(x,\omega)\rightarrow u^*(x)$, $x\in\overline{D}$, as $k\rightarrow \infty$, an application of Egorov's theorem yields convergence in $L^1(\partial D)$ so that the assertion follows by the triangle inequality and taking the limit $k\rightarrow\infty$ inside the integration in (\ref{eqn:currentep}).
\end{proof}
\begin{remark}
We would like to point out that the effective conductivity $\kappa^*$ is determined by the invariance principle on the whole space $\R^d$. 
\end{remark}
\begin{remark}
We would like to mention that for more general functions $f$ we can
prove that the well-posedness of the
forward problem \eqref{eqn:random_cond}, \eqref{eqn:fp2} together with the
weak convergence of $(X^{\omega, \ep}, L^{\omega, \ep}, A^{\omega, j})$
and some stopping arguments give a slightly weaker form
of~\eqref{eqn:pw}, namely that the convergence holds in $\mathcal
P$-probability.
\end{remark}

\subsection{Continuum approximation of the effective conductivity}
In this subsection, we provide the theoretical foundation for the convergence analysis of numerical homogenization methods based on simulation of the underlying diffusion process. More precisely, a rigorous convergence analysis of such a method requires a quantitative estimate that is stronger than the qualitative result (\ref{eqn:qual}), which was obtained in \cite{Lejay0} using merely the central limit theorem for martingales. We provide such a quantitative result in the following theorem by generalizing a classical argument due to Kipnis and Varadhan \cite{Kipnis}. The proof relies on new spectral bounds bounds which were obtained recently by Gloria and Otto \cite{GloriaOtto}. We refer the reader to the recent papers \cite{Egloffe,Gloria} for an analogous estimate for the discrete lattice random walk in random environment as well as to the paper \cite{Mourrat} which was the first one to use the Kipnis and Varadhan argument in order to obtain quantitative results. 

\begin{theorem}\label{thm:Kipnis}
Let $\{\kappa_{\ep}(x,\omega),(x,\omega)\in\R^d\times\Gamma\}$ be a stationary random field satisfying assumptions (A1) and (A2) from Subsection \ref{subsec:proba}. 
Then for every direction vector $\xi\in \R^d$ there exist positive constants $c_1$, $c_2$ such that
\begin{equation}\label{eqn:estimate}
\left\lvert \frac{\overline{\EW}(X_t^{\cdot}\cdot\xi)^2}{2t}-\xi\cdot\kappa^*\xi\right\rvert\leq c_1\begin{cases} \lvert\log t \rvert^{c_2} t^{-1}\quad
&d=2\\
t^{-1}\quad&d=3.
\end{cases}
\end{equation}
\end{theorem}
\begin{proof}
For fixed $\omega\in\Gamma$ let us consider the diffusion process
$X^{\omega}$ on $\R^d$ which is associated with the symmetric regular
Dirichlet form $(\E^{\omega,1},H^1(\R^d))$ on $L^2(\R^d)$ under the
measure $\mathbb{P}_0^{\omega}$. Following \cite{Kipnis}, we search for
a decomposition of the form
\begin{equation}\label{Kipnis_decomp}
X_{t}^{\omega}=M_{t}^{\omega}+R_{t}^{\omega},
\end{equation}
where $M_t^{\omega}$ is a $\mathbb{P}_0^{\omega}$-martingale and for
every direction $\xi\in\R^d$ the projected remainder
$R_{t}^{\omega}\cdot\xi$ converges to zero in
$L^2(\overline{\Gamma})$ as $t\rightarrow\infty$.

Once we have found a suitable decomposition~\eqref{Kipnis_decomp}, we first show that
\begin{equation}\label{eqn:estimate2}
t^{-1}\overline{\EW}\{(X_t^\cdot\cdot \xi)^2 - (M_t^\cdot\cdot\xi)^2\} =
t^{-1}\overline{\EW}(R_t^\cdot\cdot \xi)^2.
\end{equation}
Then we will use spectral calculus
to estimate the right-hand side of~\eqref{eqn:estimate2}.

In order to obtain the decomposition~\eqref{Kipnis_decomp}, we recall that the auxiliary problem (\ref{eqn:auxiliary}) is equivalent to the following stochastic elliptic equation in the physical space $\R^d$: 
Find $\boldsymbol{\chi}^{\xi}\in \mathcal{V}^2_{\text{pot}}(\Gamma)$ such that for $\mathcal{P}$-a.e. $\omega\in\Gamma$, the corresponding potential $\eta^{\xi}:\R^d\times\Gamma\rightarrow\R$ 
is in $C(\R^d)\cap H^1_{\text{loc}}(\R^d)$ as a function of $x$, satisfies $\eta^{\xi}(0,\omega)= 0$ and 
\begin{equation}\label{eqn:corrector}
-\nabla\cdot\kappa(\cdot,\omega)(\xi+\nabla\eta^{\xi}(\cdot,\omega))=0\quad\text{in }\R^d.
\end{equation}
Let us define the function 
$$\phi:\R^d\times\Gamma\rightarrow\R^d, \quad\phi(x,\omega):= x+\eta(x,\omega)-\eta(0,\omega),$$ 
where $\eta:=(\eta^{e_1},...,\eta^{e_d})^T$ and $\eta^{e_i}$,
$i=1,...,d$, denotes the potential corresponding to the solution to
the auxiliary problems for the coordinate directions. As the transition
density kernel of $X^{\omega}$ is jointly H\"older-continuous and
$\phi_i(\cdot,\omega)\in C(\R^d)\cap H^1_{\text{loc}}(\R^d)$,
$i=1,...,d$, for $\mathcal{P}$-a.e. $\omega\in\Gamma$, the Fukushima
decomposition of $\phi_i(X_t^{\omega},\omega)$ holds for every starting
point $x\in \R^d$ rather than quasi-every $x\in \R^d$. A straightforward
computation using the fact that $\eta$ is defined via the auxiliary
problem yields that the continuous additive functional of zero energy in
the Fukushima decomposition vanishes so that
\begin{equation*}
\phi_i(X_t^{\omega},\omega)=M_t^{\phi_i(\cdot,\omega)},\quad i=1,...,d.
\end{equation*}
We set
\begin{equation*}
M^{\omega}_t:=(M_t^{\phi_1(\cdot,\omega)},...,M_t^{\phi_d(\cdot,\omega)})^T\quad\text{and}\quad R^{\omega}_t:= -\eta(X_t^{\omega},\omega)+\eta(0,\omega)
\end{equation*}
and consider the quantity
\begin{equation}\label{eqn:quantity}
\overline{\EW}(X_t\cdot\xi)^2=\overline{\EW}(M_t^{\cdot}\cdot\xi)^2+\overline{\EW}(R_t^{\cdot}\cdot\xi)^2+2\overline{\EW}(M_t^{\cdot}\cdot\xi)(R_t^{\cdot}\cdot\xi).
\end{equation}
By computing the predictable quadratic variation of the martingale
additive functional we obtain for all $t\geq 0$ and a.e. $x\in\R^d$
\begin{eqnarray*}
\EW_x^{\omega}(M_t^{\omega}\cdot\xi)^2
=\EW_x^{\omega}\Big(\int_0^t2(\xi+\nabla\eta^{\xi}(X_s,\omega))\cdot\kappa(X_s,\omega)(\xi+\nabla\eta^{\xi}(X_s,\omega))\dd s\Big).\end{eqnarray*}
By the stationarity of $\{\nabla\eta^{\xi}(x,\omega),(x,\omega)\in\R^d\times\Gamma\}$ with respect to $\mathcal{P}$ we have thus
\begin{equation*}
\overline{\EW}(M_t^{\omega}\cdot\xi)^2= 2\xi\cdot\kappa^*\xi t\quad\text{for all }t\geq 0.
\end{equation*}
Moreover, as in \cite{DeMasi}, it follows that the last term on the
right-hand side of (\ref{eqn:quantity}) vanishes. This can be seen by studying the so-called \emph{environment seen by the particle} process, i.e., the stochastic process defined by
\begin{equation*}
Y_t^{\omega}:=\begin{cases}
\boldsymbol{\Theta}_{X^{\omega}_t}\omega&\quad t>0\\
\omega&\quad t=0.
\end{cases}
\end{equation*}
$\{Y_t^{\omega},t\geq 0\}$ is a stationary process with respect to the annealed measure $\overline{\mathbb{P}}$, i.e., 
for every finite collection of times $t^{(i)}$, $i=1,...,k$, the joint distribution of $Y_{t^{(1)}+h},...,Y_{t^{(k)}+h}$ under $\overline{\mathbb{P}}$ does not depend on $h\geq 0$. It is well known that the underlying dynamical system $\{\boldsymbol{\Theta}_x,x\in\R^d\}$ defines a $d$-parameter group $\{\boldsymbol{S}_x,x\in\R^d\}$ of unitary operators on $L^2(\Gamma)$ by $\boldsymbol{S}_x\boldsymbol{\psi}(\omega):=\boldsymbol{\psi}(\boldsymbol{\Theta}_x\omega)$ and this group is strongly continuous, cf. \cite{Lejay0}. Its $d$ infinitesimal generators $(\boldsymbol{D}_1,\D(\boldsymbol{D}_1)),...,(\boldsymbol{D}_d,\D(\boldsymbol{D}_d))$ are given by
\begin{equation*}
\boldsymbol{D}_i\boldsymbol{\psi}=\lim_{h\rightarrow 0+}\frac{\boldsymbol{S}_{he_i}\boldsymbol{\psi}-\boldsymbol{\psi}}{h},\quad i=1,...,d,
\end{equation*}
for all $\boldsymbol{\psi}\in L^2(\Gamma)$ such that the limit exists. These operators are closed and densely defined. We denote $\boldsymbol{D}:=(\boldsymbol{D}_1,...,\boldsymbol{D}_d)^T$ and introduce the infinitesimal generator $(\boldsymbol{\mathcal L},\D(\boldsymbol{\mathcal L}))$ on $L^2(\Gamma)$ of the environment viewed by the particle process, that is, the non-negative definite self-adjoint operator $\boldsymbol{\mathcal{L}}:=-\boldsymbol{D}\cdot\boldsymbol{\kappa}\boldsymbol{D}$ on $L^2(\Gamma)$. Note that due to the fact that the trajectories of the random conductivity field are in general not differentiable, one has to use Dirichlet form theory in order to give a precise meaning to computations involving $\boldsymbol{\mathcal{L}}$. We refer the reader to the work \cite{Lejay0} where this has been carried out in some detail. By the
self-adjoinedness of $\boldsymbol{\mathcal L}$, the law of the
environment as viewed from the particle process under
$\overline{\mathbb{P}}$ is invariant with respect to time reversal and
$M^{\omega}$ is odd by \cite[Corollary 2.1]{Fitzsimmons}, i.e., it
changes its sign under time reversal, whereas $R^{\omega}$, which is the
zero energy part of the Fukushima decomposition (\ref{Kipnis_decomp}),
is even by \cite[Theorem 2.1]{Fitzsimmons}. Thus, we notice that the identity~\eqref{eqn:estimate2} holds, as claimed.

We will next show that the estimate~\eqref{eqn:estimate} follows from the
following spectral representation
\begin{equation}\label{eqn:spectral}
\overline{\mathbb{\EW}}(R^{\cdot}_t\cdot\xi)^2=2\int_0^{\infty}(1-e^{-\lambda t})\lambda^{-2}\dd(E_{\lambda}\boldsymbol{v}^{\xi},\boldsymbol{v}^{\xi}),
\end{equation}
where $\{E_{\lambda},\lambda\in\R\}$ is the unique spectral family given 
by the spectral theorem 
such that
$\boldsymbol{\mathcal{L}}=\int_{0}^{\infty}\lambda\dd E_{\lambda}$ and
the function
$\boldsymbol{v}^{\xi}:=\boldsymbol{D}\cdot\boldsymbol{\kappa}\xi\in
L^2(\Gamma)$.
Indeed, given the formula~\eqref{eqn:spectral} for the projected
remainder and due to the assumption (A2), we
can now exploit the following optimal
estimate from \cite{GloriaNeukamm, GloriaOtto}: For all $0<\gamma\leq
1$, there exists a positive constant $c$ such that 
\begin{equation*}
\int_0^{\gamma}\dd(E_{\lambda}\boldsymbol{v}^{\xi},\boldsymbol{v}^{\xi})\leq c \gamma^{d/2+1}.
\end{equation*}

More precisely, we split the integral (\ref{eqn:spectral}) into three
parts, the first ranging from $0$ to $t^{-1}$, the second from $t^{-1}$
to $1$ and the third from $1$ to $\infty$, respectively when $t > 1$. For the latter
we have the trivial estimate 
\begin{equation*}
\int_1^{\infty}\dd(E_{\lambda}\boldsymbol{v}^{\xi},\boldsymbol{v}^{\xi})\leq\int_0^{\infty}\dd(E_{\lambda}\boldsymbol{v}^{\xi},\boldsymbol{v}^{\xi})=\mathbb{M}\{(\boldsymbol{v}^{\xi})^2\}.
\end{equation*}
The first part is bounded by a positive constant as well, namely by
\begin{eqnarray*}
\int_0^{t^{-1}}t\lambda^{-1}\dd(E_{\lambda}\boldsymbol{v}^{\xi},\boldsymbol{v}^{\xi})&=&t\int_0^{t^{-1}}\int_{\lambda}^{\infty}\alpha^{-2}\dd\alpha\dd(E_{\lambda}\boldsymbol{v}^{\xi},\boldsymbol{v}^{\xi})\\
&=&t\int_0^{\infty}\alpha^{-2}\int_0^{\alpha\wedge t^{-1}}\dd(E_{\lambda}\boldsymbol{v}^{\xi},\boldsymbol{v}^{\xi})\\
&\leq&c t\int_0^{\infty}\alpha^{-2}(\alpha\wedge t^{-1})^{d/2+1}\dd\alpha.
\end{eqnarray*}
Similarly, the second part can be estimated by
\begin{eqnarray*}
\int_{t^{-1}}^1\lambda^{-2}\dd(E_{\lambda}\boldsymbol{v}^{\xi},\boldsymbol{v}^{\xi})&=&2\int_{t^{-1}}^1\int_{\lambda}^{\infty}\alpha^{-3}\dd\alpha\dd(E_{\lambda}\boldsymbol{v}^{\xi},\boldsymbol{v}^{\xi})\\
&\leq&2c\int_{t^{-1}}^{\infty}\alpha^{-3}(\alpha\wedge 1)^{d/2+1}\dd\alpha,
\end{eqnarray*}
which diverges logarithmically for $d=2$ and is bounded by a positive
constant for $d=3$. Therefore, combining these computations with the
identities~\eqref{eqn:estimate2} and~\eqref{eqn:spectral} the
estimate~\eqref{eqn:estimate} follows.

It remains to prove the spectral representation~\eqref{eqn:spectral}. 
In order to take advantage of the
spectral theorem, we would like to express the projected remainder
in the form $\mathbb{M}\psi_1(\boldsymbol{\mathcal{L}})\boldsymbol{v}^{\xi}
\psi_2(\boldsymbol{\mathcal{L}})\boldsymbol{v}^{\xi}$ for some bounded continuous
functions $\psi_1$ and $\psi_2$. However, for this we would need the components of
$\eta$ to be stationary, which is not the case. Furthermore, we
would want to use $\psi_1(x) = x^{-1}$ which is unbounded at zero.

Inspired by the computations in \cite{Pa1}, we can find a remedy for both of these
obstructions; namely we consider the function
$$R^{\omega,\delta}_t:=-\eta_{\delta}(X_t^{\omega},\omega)+\eta_{\delta}(0,\omega),$$ 
where $\eta_{\delta}$ is defined in analogy to $\eta$ with the difference that it corresponds to a different auxiliary problem, modified by a zero order term. This modified auxiliary problem reads as follows: 
Find $\eta^{\xi}_{\delta}(\cdot,\omega)\in C(\R^d)\cap
H^1_{\text{loc}}(\R^d)$, $\delta>0$ such that for $\mathcal{P}$-a.e.
$\omega\in\Gamma$, the random field
$\{\eta^{\xi}_{\delta}(x,\omega),(x,\omega)\in\R^d\times\Gamma\}$ is
stationary with respect to $\mathcal{P}$ with
$\mathbb{M}\eta_{\delta}^{\xi}(x,\cdot)=0$ for every $x\in\R^d$ and
satisfies
\begin{equation*}
\delta\eta_{\delta}^{\xi}(\cdot,\omega)-\nabla\cdot\kappa(\cdot,\omega)(\xi+\nabla\eta^{\xi}_{\delta}(\cdot,\omega))=0\quad\text{in }\R^d 
\end{equation*}
for $\mathcal{P}$-a.e. $\omega\in\Gamma$.
We refer to  \cite{Pa1} for the proof of existence and uniqueness of
$\eta^{\xi}_{\delta}$, $\delta>0$. Note that
$\eta_{\delta}^{\xi}(0,\omega)\neq 0$ in general so that we have for
the projected modified remainder the expression
\begin{equation}\label{eqn:modified1}
\overline{\mathbb{\EW}}(R^{\cdot,\delta}_t\cdot\xi)^2=\overline{\mathbb{\EW}}(\eta_{\delta}^{\xi}(X_t^{\cdot}))^2
-2\overline{\mathbb{\EW}}\eta_{\delta}^{\xi}(X_t^{\cdot},\cdot)\eta_{\delta}^{\xi}(0,\cdot)+\overline{\mathbb{\EW}}(\eta_{\delta}^{\xi}(0,\cdot))^2.
\end{equation} 

The equivalent formulation on $L^2(\Gamma)$ of the modified auxiliary problem for the direction
$\xi\in\R^d$ reads as follows: Find the
unique solution $\boldsymbol{\eta}_{\delta}^{\xi}\in L^2(\Gamma)$ of
the elliptic equation
\begin{equation*}
\delta\boldsymbol{\eta}_{\delta}^{\xi}-\boldsymbol{D}\cdot\boldsymbol{\kappa}(\xi+\boldsymbol{D}\boldsymbol{\eta}_{\delta}^{\xi})=0\quad\text{in }\Gamma.
\end{equation*}
In particular, the function $\boldsymbol{v}^{\xi}$ was chosen such that
$\boldsymbol{v}^{\xi}=(\delta+\boldsymbol{\mathcal{L}})\boldsymbol{\eta}^{\xi}_{\delta}
= \boldsymbol{\mathcal{L}}\boldsymbol{\eta}^{\xi}$.

We will now only need to show that
the modified remainder for fixed $\delta$ and fixed $t$ can be written
in the form $\mathbb{M}\psi_1(\boldsymbol{\mathcal{L}})\boldsymbol{v}^{\xi}
\psi_2(\boldsymbol{\mathcal{L}})\boldsymbol{v}^{\xi}$, 
where $\psi_1(x) =
2(x+\delta)^{-1}$ and $\psi_2(x) = 2e^{-tx}(x+\delta)^{-1}$. In fact, the
introduction of the zero order perturbation removes the singularity 
coming from the term $x^{-1}$.

The first term on the right-hand side of~\eqref{eqn:modified1} may equivalently be written as
\begin{equation*}
\overline{\EW}(\eta_{\delta}^{\xi}(X_t^{\cdot},\cdot))^2=\overline{\EW}(\eta_{\delta}^{\xi}(0,Y_t^{\cdot}))^2=\mathbb{M}(\eta_{\delta}^{\xi}(0,\cdot))^2,
\end{equation*}
where we have used the stationarity of the environment as viewed from
the particle with respect to $\overline{\mathbb{P}}$ and the
stationarity of the random field
$\{\eta_{\delta}^{\xi}(x,\omega),(x,\omega)\in\R^d\times\Gamma\}$ with
respect to $\mathcal{P}$, respectively.
Therefore, treating the second term on the right-hand side
of~\eqref{eqn:modified1} we obtain 
\begin{equation}\label{eqn:key}
\overline{\mathbb{\EW}}(R^{\cdot,\delta}_t\cdot\xi)^2=2\mathbb{M}(\eta_{\delta}^{\xi}(0,\cdot))^2-2\mathbb{M}\eta_{\delta}^{\xi}(0,\cdot)T_t^{\cdot}\eta_{\delta}^{\xi}(0,\cdot),
\end{equation}
where $\{T_t^{\omega},t\geq 0\}$ denotes the strongly continuous
semigroup on $L^2(\R^d)$ associated with $(\E^{\omega,1},H^1(\R^d))$
which satisfies 
$$\EW_0^{\omega}\eta_{\delta}^{\xi}(X_t^{\omega},\omega)=T_t^{\omega}\eta_{\delta}^{\xi}(0,\omega).$$ 

Going back to (\ref{eqn:key}), respectively the corresponding identity
on $L^2(\Gamma)$, we have thus for the first term on the right-hand side
\begin{equation*}
2\mathbb{M}(\boldsymbol{\eta}_{\delta}^{\xi})^2=2\mathbb{M}(\delta+\boldsymbol{\mathcal L})^{-1}\boldsymbol{v}^{\xi}(\delta+\boldsymbol{\mathcal L})^{-1}\boldsymbol{v}^{\xi}=2\int_0^{\infty}(\delta+\lambda)^{-2}\dd(E_{\lambda}\boldsymbol{v}^{\xi},\boldsymbol{v}^{\xi}),
\end{equation*}
whereas the second term may be written in the form 
\begin{equation*}
2\mathbb{M}(\delta+\boldsymbol{\mathcal L})^{-1}\boldsymbol{v}^{\xi}e^{-t\boldsymbol{\mathcal{L}}}(\delta+\boldsymbol{\mathcal L})^{-1}\boldsymbol{v}^{\xi}=2\int_0^{\infty}(\delta+\lambda)^{-2}e^{-t\lambda}\dd(E_{\lambda}\boldsymbol{v}^{\xi},\boldsymbol{v}^{\xi}).
\end{equation*}
Altogether we have obtained
\begin{equation*}
\overline{\EW}(R_t^{\cdot,\delta}\cdot\xi)^2=2\int_0^{\infty}(1-e^{-\lambda t})(\delta+\lambda)^{-2}\dd(E_{\lambda}\boldsymbol{v}^{\xi},\boldsymbol{v}^{\xi})
\end{equation*}
and the right-hand side of this equality converges as $\delta\rightarrow 0$ so that
the spectral representation~\eqref{eqn:spectral} holds.
\end{proof}
We would like to remark that it appears that the spectral properties of
operators with stationary coefficients and the stationarity of the
corresponding fields (as in the proof of Theorem~\ref{thm:Kipnis}) might
be very closely related in general. However, analysing this potential connection
that is touched by the work~\cite{Pa1}
is left for the future studies.
\section{Conclusion}\label{section6}
We have derived Feynman-Kac formulae for the forward problem of electrical im-pedance tomography and studied the interconnection between these formulae and stochastic homogenization. Using the properties of the underlying diffusion processes and some new spectral estimates from \cite{Gloria,GloriaNeukamm} we have then obtained a bound on the speed of convergence of the projected mean-square displacement of the processes. These results provide the theoretical foundation for the development of new scalable continuum Monte Carlo homogenization schemes. 

Both, the homogenization of the forward model for the complete electrode model and the stochastic numerical approximation of the effective conductivity have direct applications in EIT anomaly detection problems for random heterogeneous background media, cf. \cite{Simon}.

%For acknowledgements section, please don't number the section, please begin it with \section*{Acknowledgements}
\section*{Acknowledgments} 
The research of M.~Simon was supported by the Deutsche
Forschungsgemeinschaft (DFG) under grant HA 2121/8 -1 583067.
This work is part of M.~Simon's Ph.D thesis, who would like to express his gratitude to his 
advisor Prof.~Martin Hanke for his guidance and continuous support. He
would also like to thank Prof.~Lassi P\"{a}iv\"{a}rinta for the kind
invitation to the Department of Mathematics and Statistics at the
University of Helsinki, where part of the work was carried out.
The research of P.~Piiroinen was supported by Academy of Finland (AF) under Finnish
Centre of Excellence in Inverse Problems Research 2012--2017,
decision number 250215. He has also been supported by an AF project,
decision number 141075.
Both authors would like to thank Prof. Antoine Gloria and Prof. Elton Hsu for carefully reading parts of this manuscript and for their insightful comments and suggestions.

% You may incorporate your references as follows in your main tex file.
% Using BibTex is not recommended but can be handled.

\end{document}